\documentclass[12pt, reqno]{amsart}
\usepackage{amsmath, amstext, amsbsy, amssymb, amscd}
\usepackage{amsmath}
\usepackage{amsxtra}
\usepackage{amscd}
\usepackage{amsthm}
\usepackage{amsfonts}
\usepackage{amssymb}
\usepackage{eucal}
\usepackage{color}
\usepackage{youngtab}

\input prepictex
\input pictex
\input postpictex

\newtheorem{theorem}{Theorem}[section]

\newtheorem{lemma}[theorem]{Lemma}
\newtheorem{proposition}[theorem]{Proposition}
\newtheorem{corollary}[theorem]{Corollary}
\theoremstyle{definition}
\newtheorem{definition}[theorem]{Definition}
\newtheorem{example}[theorem]{Example}
\newtheorem{remark}[theorem]{Remark}
\definecolor{A}{rgb}{.75,1,.75}

\numberwithin{equation}{section}

\newcommand{\ad}{\operatorname{ad}}
\newcommand{\C}{ \mathbb C }

\newcommand{\Di}{D_i^{(r)}}

\newcommand{\Ei}{E_i^{(r)}}
\newcommand{\Fi}{F_i^{(r)}}

\newcommand{\Daij}{D_{a;i,j}^{(r)}}
\newcommand{\Dpaij}{D_{a;i,j}^{\prime(r)}}
\newcommand{\Eaij}{E_{a;i,j}^{(r)}}
\newcommand{\Faij}{F_{a;i,j}^{(r)}}
\newcommand{\del}{\delta}

\newcommand{\End}{\operatorname{End}}
\newcommand{\W}{\mathcal{W}}

\newcommand{\ev}{\operatorname{ev}}

\newcommand{\gl}{\mathfrak{gl}}

\newcommand{\g}{\mathfrak{g}}

\newcommand{\gr}{\operatorname{gr}}
\newcommand{\glMN}{\mathfrak{gl}_{M|N}}

\newcommand{\id}{\operatorname{id}}

\newcommand{\la}{\lambda}

\newcommand{\pa}{\overline}
\newcommand{\tp}{\operatorname{pa}}
\newcommand{\tpa}{\check{\operatorname{pa}}}
\newcommand{\pr}{\operatorname{pr}}
\newcommand{\col}{\text{col}}
\newcommand{\row}{\text{row}}

\newdimen\Hoogte    \Hoogte=12pt
\newdimen\Breedte   \Breedte=12pt
\newdimen\Dikte     \Dikte=0.5pt

\newenvironment{Young}{\begingroup
       \def\vr{\vrule height0.8\Hoogte width\Dikte depth 0.4\Hoogte}
       \def\fbox##1{\vbox{\offinterlineskip
                    \hrule height\Dikte
                    \hbox to \Breedte{\vr\hfill##1\hfill\vr}
                    \hrule height\Dikte}}
       \vbox\bgroup \offinterlineskip \tabskip=-\Dikte \lineskip=-\Dikte
            \halign\bgroup &\fbox{##\unskip}\unskip  \crcr }
       {\egroup\egroup\endgroup}
\def\Diagram#1{\relax\ifmmode\vcenter{\,\begin{Young}#1\end{Young}\,}\else%
              $\vcenter{\,\begin{Young}#1\end{Young}\,}$\fi}

\begin{document}
\title[On shifted super Yangians and a class of finite $W$-superalgebras] {On shifted super Yangians and a class of finite $W$-superalgebras}

\author[Yung-Ning Peng]{Yung-Ning Peng}
\address{Institute of Mathematics, Academia Sinica,
Taipei City, Taiwan, 10617, and
Department of Mathematics, National Central University, Chung-Li, Taiwan, 32054} \email{ynp@math.sinica.edu.tw; ynp@math.ncu.edu.tw}

\begin{abstract}
Let $e$ be an even nilpotent element, satisfying certain restrictions on its Jordan type, in a general linear Lie superalgebra.
We study the finite $W$-superalgebra $\W_e$ associated to such an $e$, and a realization of $\W_e$ in terms of a quotient of a shifted super Yangian is established.
\end{abstract}

\maketitle

\setcounter{tocdepth}{1}
\tableofcontents

\section{Introduction}
A finite $W$-algebra is an associative algebra constructed from a pair $(\g,e)$, where $\g$ is a finite dimensional semisimple or reductive Lie algebra and $e$ is a nilpotent element of $\g$. In the extreme case where $e=0$, the corresponding finite $W$-algebra is isomorphic to $U(\g)$, the universal enveloping algebra. In the other extreme case where $e$ is the principal (also called regular) nilpotent element $e$, Kostant proved that the associated finite $W$-algebra is isomorphic to the center of the universal enveloping algebra $U(\g)$ (cf. \cite{Ko}). In recent decades, there were many new developments on $W$-algebras; we refer to the survey papers \cite{Lo} and \cite{Wa} for details.

On the other hand, the Yangians, defined by Drinfeld in 1983, are certain non-commutative Hopf algebras that are important examples of quantum groups. They were used to generate the rational solutions of the Yang-Baxtor equation; see the book \cite{Mo} for more details and further applications of the Yangians.

The connection between Yangians and finite $W$-algebras of type A associated to a {\em rectangular} nilpotent element $e$ was first noticed by Ragoucy and Sorba in \cite{RS}. The term ``rectangular" means that the Jordan blocks of $e$ are all of the same size.
Brundan and Kleshchev in \cite{BK2} generalized the result to an arbitrary nilpotent $e$ by a different approach. As a consequence, a realization of finite $W$-algebra of type A as a quotient of a so-called shifted Yangian is obtained, and this provides a powerful tool for the study of finite $W$-algebras.

In this article, we establish such a connection between finite $W$-superalgebras and super Yangians explicitly in type A where the Jordan type of $e$ satisfies a certain condition (\ref{crucialhypo}). Let $Y_{m|n}$ denote the super Yangian of the general linear Lie superalgebra $\gl_{m|n}$. In fact, such a connection was firstly obtained in \cite{BR} when $e\in\glMN$ is rectangular; see also \cite{Pe2}. Their result shows that the finite $W$-superalgebra associated to a rectangular $e\in\glMN$ is isomorphic to the truncated super Yangian $Y_{m|n}^\ell$, which is a certain quotient of the super Yangian $Y_{m|n}$. Here the indices $m$ and $n$ are determined by the number of Jordan blocks of the rectangular $e$ and $\ell$ is the size of its Jordan block.

In a more recent paper \cite{BBG}, the connection between the finite $W$-superalgebra associated to an $e\in\glMN$ and $Y_{1|1}$ is developed. It corresponds to the case when the nilpotent element $e\in\glMN$ is principal. Our main result (Theorem \ref{main}) is to establish an isomorphism of superalgebras between the truncated shifted super Yangian for $\gl_{1|n}$ and a finite $W$-superalgebra.

Let us explain our approach, which is roughly generalizing the argument in \cite{BK2} to the general linear Lie superalgebras. Firstly we give a presentation of the shifted super Yangian, denoted by $Y_{1|n}(\sigma)$, which is a subalgebra of $Y_{1|n}$ associated to a matrix $\sigma$. The set of generators is a certain subset, determined by $\sigma$, of the generators of $Y_{1|n}$. The defining relations are modified from the defining relations of $Y_{1|n}$ according to $\sigma$ as well. Then we quotient out a certain ideal to obtain the truncated super Yangian $Y_{1|n}^\ell(\sigma)$. One may naively think that those generators with degree higher than $\ell$ vanish in $Y_{1|n}^\ell(\sigma)$.

Next we introduce certain combinatorial objects called $pyramids$; see \cite{EK}, \cite{Ho}. This gives a nice way to record the necessary information (that is, $\sigma$ and $\ell$) to define $Y_{1|n}^\ell(\sigma)$ by a diagram. For example, 
$$
\ell=4,\,\,
\sigma = \left(\begin{array}{l|ll}
0&1&1\\
\hline
0&0&0\\
1&1&0
\end{array}\right)
\qquad\longleftrightarrow \qquad
\pi = 
{\begin{picture}(90, 35)%
\put(15,-20){\line(1,0){60}}
\put(15,-5){\line(1,0){60}}
\put(30,10){\line(1,0){45}}
\put(30,25){\line(1,0){30}}
\put(15,-20){\line(0,1){15}}
\put(30,-20){\line(0,1){45}}
\put(45,-20){\line(0,1){45}}
\put(60,-20){\line(0,1){45}}
\put(75,-20){\line(0,1){30}}
\put(18,-15){$-$}\put(33,-15){$-$}\put(48,-15){$-$}\put(63,-15){$-$}
\put(33,0){$-$}\put(48,0){$-$}\put(63,0){$-$}
\put(33,15){$+$}\put(48,15){$+$}
\end{picture}}
$$

The merit of using a pyramid is that one may obtain a nilpotent element $e$ and a semisimple element $h$ that determine a finite $W$-superalgebra. In our example above, 
\begin{align*}
e=e_{\pa{1}\,\pa{2}}+e_{24}+e_{46}+e_{13}+e_{35}+e_{57},\\
h=\text{diag}(1,-1,3,1,1,-1,-1,-3,-3),
\end{align*}
where $e_{ij}$ means the elementary matrix in $\gl_{2|7}$; see \textsection 7 for detail. 

Therefore, given a pyramid $\pi$, we simultaneously obtain a truncated super Yangian $Y_{\pi}$, which is $Y_{1|n}^\ell(\sigma)$ for some appropriate choice of $\ell$ and $\sigma$, and a finite $W$-superalgebra $\W_{\pi}$ associated to a certain nilpotent $e$ determined by the pyramid $\pi$. Our main result is that there exists an isomorphism of superalgebras between $Y_\pi$ and $\W_\pi$. 

In \textsection 8, we introduce the notion of {\em super height} so that one may explicitly write down certain distinguished elements in $\W_\pi$ according to the diagram $\pi$. Eventually, we prove that the map sending the generators of $Y_\pi$ into these distinguished elements is an isomorphism of filtered superalgebras by induction on $\ell$, the number of boxes of the base of $\pi$. As a consequence, a presentation of the finite $W$-superalgebra $\W_\pi$ is obtained.

The general case, which means the even nilpotent $e\in\glMN$ could be arbitrary, is highly challenging and requires new presentations of the super Yangian that are unknown yet; see Remark \ref{gencase}.

This article is organized as follows. In \textsection 2, we define the shifted Yangian $Y_{1|n}(\sigma)$ and prove a PBW theorem for it. In \textsection 3, we introduce the notion of parabolic presentations for $Y_{1|n}(\sigma)$ as in \cite{Pe1}. An important consequence is that we may write down an explicit formula for the so-called baby comultiplications in \textsection 4. In \textsection 5, we introduce the canonical filtration of $Y_{1|n}(\sigma)$. In the end it corresponds to the Kazhdan filtration of finite $W$-superalgebras. Then we define the truncated shifted Yangian $Y_{1|n}^\ell(\sigma)$ in \textsection 6 as a quotient of $Y_{1|n}(\sigma)$, and prove that $Y_{1|n}^\ell(\sigma)$ shares many nice properties of $Y_{1|n}(\sigma)$ such as PBW bases and baby comultiplications. Next we switch our attention to finite $W$-superalgebras. In \textsection 7 we give the definition of the finite $W$-superalgebra with respect to an even good $\mathbb{Z}$-grading. Then we use {\em pyramids} as a tool to encode the information needed to define a finite $W$-superalgebra. Moreover, we explain how to read off a truncated shifted Yangian $Y_{1|n}^\ell(\sigma)$ from a given pyramid $\pi$. In \textsection 8, we give explicitly the formulae for some elements in $U(\mathfrak{p})$ that eventually can be identified as generators of our finite $W$-superalgebra. Our main theorem is stated and proved in \textsection 9.

{\em Notation:} In this article the underlying field is always $\mathbb{C}$. A superalgebra means an associative $\mathbb{Z}_2$-graded algebra. The parity of a homogeneous element $x$ is denoted by $|x|$. For homogeneous elements $x$ and $y$ in a superalgebra $A=A_{\pa{0}}\oplus A_{\pa{1}}$, their supercommutator is $[x,y]:= xy-(-1)^{|x||y|}yx$. We say $x$ and $y$ supercommute if $[x,y]=0$. For homogeneous $x_1,\ldots,x_t\in A$, an ordered supermonomial in $x_1,\ldots,x_t$ means a monomial of the form $x_1^{i_1}\cdots x_t^{i_t}$ for some $i_1,\ldots,i_t\in \mathbb{Z}_{\geq 0}$ and $i_j\leq 1$ if $x_j$ is odd.

\section{Shifted super Yangian of $\gl_{1|n}$}
In this section, we recall the definition of the super Yangian $Y_{1|n}$. Moreover, we define a certain superalgebra called the shifted super Yangian, which turns out to be a subalgebra of $Y_{1|n}$.

The super Yangian $Y_{1|n}$, which was introduced in \cite{Na}, is the associative $\mathbb{Z}_2$-graded algebra (i.e., superalgebra) over $\mathbb{C}$ with generators
\[
\left\lbrace t_{ij}^{(r)}\,| \; 1\le i,j \le n+1; r\ge 0\right\rbrace,
\] where $t_{ij}^{(0)}:=\delta_{ij}$
and defining relations
\begin{equation}\label{Nadef}
[t_{ij}^{(r)}, t_{hk}^{(s)}] = (-1)^{\tp(i)\tp(j) + \tp(i)\tp(h) + \tp(j)\tp(h)}
\sum_{t=0}^{\mathrm{min}(r,s) -1} \Big( t_{hj}^{(t)} t_{ik}^{(r+s-1-t)} - t_{hj}^{(r+s-1-t)}t_{ik}^{(t)} \Big),
\end{equation}
where $\tp(i) = 0$ if $i=1$ and $\tp(i)=1$ otherwise. The bracket in (\ref{Nadef}) is understood as a supercommutator. For $r>0$, the element $t_{ij}^{(r)}$ is defined to be an odd element if $\tp(i)+\tp(j)\equiv 1$ (mod 2) and an even element if $\tp(i)+\tp(j)\equiv 0$ (mod 2).

The elements $\{t_{ij}^{(r)}\}$ are called RTT generators while the defining relations (\ref{Nadef}) are called RTT relations. Next we use an alternate presentation of $Y_{1|n}$ to define shifted Yangians that are not easily obtained from the classical definition.

Let $\sigma$ be an $(n+1)\times(n+1)$ matrix $(s_{i,j})_{1\leq i,j\leq n+1}$ where the entries are non-negative integers satisfying that
\begin{equation}\label{sijk}
s_{i,j} + s_{j,k} = s_{i,k},
\end{equation}
whenever $|i-j|$+$|j-k|$=$|i-k|$. Immediately, we have $s_{1,1}=s_{2,2}=\cdots=s_{n+1,n+1}= 0$, and $\sigma$ is determined by the upper diagonal entries $s_{1,2}, s_{2,3},\ldots, s_{n,n+1}$ and the lower diagonal entries $s_{2,1}, s_{3,2},\ldots, s_{n+1,n}$. In addition, we will add lines to emphasize the parities. Such a matrix will be called a shift matrix. For example, the following matrix is a shift matrix:
\[
\sigma = \left(\begin{array}{l|llll}
0&0&1&1&2\\
\hline
0&0&1&1&2\\
1&1&0&0&1\\
3&3&2&0&1\\
4&4&3&1&0
\end{array}\right).
\]

\begin{definition}
The shifted Yangian of $\gl_{1|n}$ associated to $\sigma$, denoted by $Y_{1|n}(\sigma)$, is the superalgebra over $\C$ generated by the following elements
\[\lbrace \Di| 1\leq i\leq n+1, r\geq 0\rbrace,\]
\[\lbrace \Ei| 1\leq i\leq n, r> s_{i,i+1}\rbrace,\]
\[\lbrace \Fi| 1\leq i\leq n, r> s_{i+1,i}\rbrace,\]
subject to the following defining relations:
 \begin{align}
D_i^{(0)}&=1\\
\sum_{t=0}^{r}D_i^{(t)}D_i^{\prime(r-t)}&=\delta_{r0}\\
[D_i^{(r)},D_j^{(s)}]&=\delta_{ij}(-1)^{\tp(i)}\sum_{t=0}^{min(r,s)-1}(D_i^{(t)}D_i^{(r+s-1-t)}-D_i^{(r+s-1-t)}D_i^{(t)})
\end{align}
\begin{align}\notag
&\lefteqn{[D_i^{(r)}, E_j^{(s)}]=}\\
&\;\left\{
  \begin{array}{ll}
     \displaystyle   -\delta_{i,j}\sum_{t=0}^{r-1} D_i^{(t)}E_j^{(r+s-1-t)}
         +\delta_{i,j+1}\sum_{t=0}^{r-1}D_i^{(t)}E_j^{(r+s-1-t)}, &j \ne 1,\\[4mm]
     \displaystyle   \delta_{i,j}\sum_{t=0}^{r-1} D_i^{(t)}E_i^{(r+s-1-t)}
          +\delta_{i,j+1}\sum_{t=0}^{r-1}D_i^{(t)}E_j^{(r+s-1-t)}, &j = 1,
 \end{array}\right.\\[4mm]\notag
&\lefteqn{[D_i^{(r)}, F_j^{(s)}]=}\\
&\left\{
  \begin{array}{ll}
  \displaystyle \delta_{i,j}\sum_{t=0}^{r-1}
  F_j^{(r+s-1-t)}D_i^{(t)}
  -\delta_{i,j+1}\sum_{t=0}^{r-1}F_j^{(r+s-1-t)}D_i^{(t)}, &j \ne 1,\\[4mm]
   \displaystyle   -\delta_{i,j}\sum_{t=0}^{r-1}F_j^{(r+s-1-t)}D_i^{(t)}
      -\delta_{i,j+1}\sum_{t=0}^{r-1}F_j^{(r+s-1-t)}D_i^{(t)}\,, &j = 1,
\end{array}\right.\\[4mm]\notag
\end{align}
\begin{equation}\label{1EE}
[E_{i}^{(r)}, E_{i}^{(s)}]=
 \left\{
  \begin{array}{ll}
    \displaystyle      -\sum_{t=1}^{s-1}E_{i}^{(t)}E_{i}^{(r+s-1-t)}
          +\sum_{t=1}^{r-1}E_{i}^{(t)}E_{i}^{(r+s-1-t)},& i\neq 1,\\[5mm]
     \displaystyle     0,& i=1,\\[3mm]
  \end{array}\right.\\
\end{equation}
\begin{equation}\label{1FF}
[F_{i}^{(r)}, F_{i}^{(s)}]=
\left\{
  \begin{array}{ll}
   \displaystyle -\sum_{t=1}^{r-1}F_{i}^{(r+s-1-t)}F_{i}^{(t)}
   +\sum_{t=1}^{s-1}F_{i}^{(r+s-1-t)}F_{i}^{(t)},& i\neq 1,\\[5mm]
   \displaystyle 0, & i=1,\\[3mm]
  \end{array}\right.\\
\end{equation}

\begin{equation}
[E_{i}^{(r)}, F_{j}^{(s)}]=
         \delta_{i,j} \sum_{t=0}^{r+s-1}D_{i+1}^{(r+s-1-t)} D^{\prime (t)}_{i}\,,
\end{equation}

\begin{align}
&[E_{i}^{(r+1)}, E_{i+1}^{(s)}]-[E_{i}^{(r)}, E_{i+1}^{(s+1)}]
=-E_{i}^{(r)}E_{i+1}^{(s)}\,,\\[3mm]\label{adm}
&[F_{i}^{(r+1)}, F_{i+1}^{(s)}]-[F_{i}^{(r)}, F_{i+1}^{(s+1)}]
=F_{i+1}^{(s)}F_{i}^{(r)}\,,\\[3mm]
&[E_{i}^{(r)}, E_{j}^{(s)}] = [F_{i}^{(r)}, F_{j}^{(s)}] = 0,
\quad\text{\;\;if\;\; $|i-j|\geq 1$ },
\end{align}

\begin{eqnarray}
&\big[E_i^{(r)},[E_i^{(s)},E_j^{(k)}]\big]+
\big[E_i^{(s)},[E_i^{(r)},E_j^{(k)}]\big]=0, &|i-j|\geq 1,\\[3mm]
&\big[F_i^{(r)},[F_i^{(s)},F_j^{(k)}]\big]+
\big[F_i^{(s)},[F_i^{(r)},F_j^{(k)}]\big]=0, &|i-j|\geq 1,
\end{eqnarray}
for all ``admissible" $i,j,k,r,s$. For example, the relation (\ref{adm}) should be understood to hold for all $1\leq i\leq n-1$, $r>s_{i+1,i}$ and $s>s_{i+2,i+1}$. The elements $\{E_{1}^{(r)} | r> s_{1,2}\}\cup\{F_{1}^{(r)}| r> s_{2,1}\}$ are the only odd generators.
\end{definition}

\begin{remark}
Note that the lower degree terms in (\ref{1EE}) and (\ref{1FF}) cancel each other. Hence the elements 
$\{ \Ei, F_{i}^{(s)} | 1\leq i\leq n, \, r\leq s_{i,i+1}, \, s\leq s_{i+1,i} \}$ will not appear in $Y_{1|n}(\sigma)$ although they show up in the defining relations.
\end{remark}

\begin{remark}
In the special case where $\sigma$ is the zero matrix, $Y_{1|n}(\sigma)=Y_{1|n}$ and the above presentation is exactly a special case of the presentation of $Y_{m|n}$ given in \cite{Go}, which is a generalization of Drinfeld's presentation. We will implicitly use this isomorphism in the remaining part of this article.
\end{remark}

Let $\Gamma$ be the map sending the generators of $Y_{1|n}(\sigma)$ 
to the elements with the same name in $Y_{1|n}$. 
We now prove that the map $\Gamma:Y_{1|n}(\sigma)\longrightarrow Y_{1|n}$ 
is an injective algebra homomorphism and hence 
$Y_{1|n}(\sigma)$ is canonically a subalgebra of $Y_{1|n}$.
As a corollary, a PBW basis for $Y_{1|n}(\sigma)$ is obtained.

 Define the $loop$ $filtration$ on $Y_{1|n}(\sigma)$
\begin{equation}\label{filt2}\notag
L_0 Y_{1|n}(\sigma) \subseteq L_1 Y_{1|n}(\sigma) \subseteq L_2 Y_{1|n}(\sigma) \subseteq \cdots
\end{equation}
by setting the degree of the generators $\Di$, $\Ei$, and $\Fi$ to be $(r-1)$ 
and $L_k Y_{1|n}(\sigma)$ to be the span of all supermonomials in the generators 
of total degree~$\leq k$ and denote the associated graded algebra by $\gr^LY_{1|n}(\sigma).$

Next, for all $1\leq i<j\leq n+1$ and $r>s_{i,j}$, we define the elements $E_{i,j}^{(r)}\in~Y_{1|n}(\sigma)$ recursively by
\begin{equation}
E_{i,i+1}^{(r)}:= \Ei, \qquad E_{i,j}^{(r)}:=-[E_{i,j-1}^{(r-s_{j,j-1})}, E_{j-1}^{(s_{j-1,j}+1)}].
\end{equation}
Similarly, for $1\leq i<j\leq n+1$ and $r>s_{j,i}$, we may define the elements $F_{j,i}^{(r)}\in Y_{1|n}(\sigma)$ by
\begin{equation}
F_{i+1,i}^{(r)}:= \Fi, \qquad F_{j,i}^{(r)}:=-[F_{j-1}^{(s_{j,j-1}+1)}, F_{j-1,i}^{(r-s_{j,j-1})}].
\end{equation}

For all $1\leq i,j\leq n+1$ and $r\geq s_{i,j}$, define
\begin{equation}\notag
e_{i,j;r} := \left\{
\begin{array}{ll}
\gr^L_{r} D_{i}^{(r+1)} &\hbox{if $i = j$,}\\
\gr^L_{r} E_{i,j}^{(r+1)} &\hbox{if $i < j$,}\\
\gr^L_{r} F_{i,j}^{(r+1)} &\hbox{if $i > j$,}
\end{array}
\right.
\end{equation}
where the elements on the right-hand side are in $\gr^LY_{1|n}(\sigma)$ of degree $r$.

Denote the Lie superalgebra $\gl_{1|n}\otimes \mathbb{C}[t]$ with basis elements 
$\lbrace e_{i,j}t^r\rbrace_{1\leq i,j\leq n+1, r\geq 0}$ by $\gl_{1|n} [t]$, which can 
be viewed as a graded Lie superalgebra by setting the degree of $e_{i,j}t^r$ to be $r$.

By assumption (\ref{sijk}), the elements $\{ e_{i,j}t^r \,|\, 1\leq i,j\leq n+1, r\geq s_{i,j}\}$ 
generate a subalgebra of $\gl_{1|n} [t]$ which we denote by $\gl_{1|n}[t](\sigma)$. 
The grading on $\gl_{1|n}[t](\sigma)$ induces a grading 
on $U\big(\gl_{1|n}[t](\sigma)\big)$, the universal enveloping algebra.

\begin{theorem}\label{UYiso}
The map $\gamma:U\big(\gl_{1|n}[t](\sigma)\big)\longrightarrow \gr^LY_{1|n}(\sigma)$ such that \[e_{i,j}t^r\longmapsto (-1)^{\tp(i)}e_{i,j;r}\] for each $1\leq i,j\leq n+1$ and $r\geq s_{i,j}$ is an isomorphism of graded superalgebras.
\end{theorem}

\begin{proof}
Following \cite[Theorem 3]{Go}, one can show that for all $1\leq i,j,h,k\leq n+1$ and $r\geq s_{i,j}$, $s\geq s_{h,k}$, the following identity holds in $\gr^LY_{1|n}$:
\begin{equation}\label{s-inj}
[e_{i,j;r},e_{h,k;s}]=(-1)^{\tp(j)}\delta_{h,j}e_{i,k;r+s}
-(-1)^{\tp(i)\tp(j)+\tp(i)\tp(h)+\tp(j)\tp(h)}\delta_{i,k}e_{h,j;r+s}.
\end{equation}
As a result, the map $\gamma$ is a well-defined homomorphism and obviously surjective. It remains to show that $\gamma$ is injective.

We first assume that $\sigma$ is the zero matrix, hence $Y_{1|n}(\sigma)=Y_{1|n}$. According to the proof of \cite[Theorem 3]{Go}, we know that the set of all ordered supermonomials in the elements $\{ e_{i,j;r} \,|\, 1\leq i,j\leq n+1, \, r\geq 0\}$ are linearly independent in $\gr^LY_{1|n}$ and $\gamma$ is an isomorphism in this special case.

In general, the map $\Gamma:Y_{1|n}(\sigma)\longrightarrow Y_{1|n}$ is a homomorphism of filtered superalgebras, which induces a map $\gr^LY_{1|n}(\sigma)\longrightarrow \gr^LY_{1|n}$ sending $e_{i,j;r}\in~\gr^LY_{1|n}(\sigma)$ to $e_{i,j;r}\in \gr^LY_{1|n}$. The previous paragraph implies that the set of ordered supermonomials in the elements $\{ e_{i,j;r} \,|\, 1\leq i,j\leq n+1, r\geq s_{i,j}\}$ are linearly independent in $\gr^LY_{1|n}(\sigma)$ and hence $\gamma$ is an isomorphism in general.
\end{proof}

\begin{remark}
In general, $\Gamma$ does not send the general parabolic elements $E^{(r+1)}_{i,j}$, $F^{(r+1)}_{j,i}$ in $Y_{1|n}(\sigma)$ 
to $E^{(r+1)}_{i,j}$, $F^{(r+1)}_{j,i}$ of  $Y_{1|n}$ if $j-i>1$.
\end{remark}

\begin{corollary}
The canonical map $\Gamma:Y_{1|n}(\sigma)\longrightarrow Y_{1|n}$ is injective.
\end{corollary}
\begin{proof}
The map $\Gamma:Y_{1|n}(\sigma)\longrightarrow Y_{1|n}$ is a filtered map and the induced map $\gr^LY_{1|n}(\sigma)\longrightarrow \gr^LY_{1|n}$ is injective. The corollary follows from induction on degree.
\end{proof}

We denote the subalgebra of $Y_{1|n}(\sigma)$ generated by all the $\Di$'s by $Y^0_{1|n}$, 
the subalgebra generated by all the $\Ei$'s by $Y^+_{1|n}(\sigma)$ and 
the subalgebra generated by all the $\Fi$'s by $Y^-_{1|n}(\sigma)$, respectively. The following corollary gives PBW bases of $Y_{1|n}(\sigma)$ and its subalgebras. 

\begin{corollary}\label{pbw1}
(1)
The set of monomials in the elements $\{D_i^{(r)}\}_{1 \leq i \leq n+1,r > 0}$ taken in some fixed order forms a basis for $Y^0_{1|n}$.\\
(2)
The set of supermonomials in the elements $\{E_{i,j}^{(r)}\}_{1 \leq i < j \leq n+1, r > s_{i,j}}$ taken in some fixed order forms a basis for $Y^+_{1|n}(\sigma)$.\\
(3)
The set of supermonomials in the elements $\{F_{j,i}^{(r)}\}_{1 \leq i < j \leq n+1, r > s_{j,i}}$ taken in some fixed order forms a basis for $Y^-_{1|n}(\sigma)$.\\
(4)
The set of supermonomials in the union of the elements listed in (1)--(3) taken in some fixed order forms a basis for $Y_{1|n}(\sigma)$.
\end{corollary}

\begin{proof}
($4$) follows from Theorem \ref{UYiso} and the PBW theorem for $U(\gl_{1|n}[t](\sigma))$. The others can be proved similarly using (\ref{s-inj}).
\end{proof}

\begin{corollary}\label{iso1}
The multiplicative map
\[Y^-_{1|n}(\sigma)\otimes Y^0_{1|n}\otimes Y^+_{1|n}(\sigma)~\longrightarrow~Y_{1|n}(\sigma)\]
is an isomorphism of vector spaces.
\end{corollary}

By the defining relations of $Y_{1|n}$, the map $\tau:Y_{1|n}\rightarrow Y_{1|n}$ defined by
\begin{equation}\label{taudef}
\tau(D_i^{(r)}) = D_i^{(r)},\quad
\tau(E_i^{(r)}) = F_i^{(r)},\quad
\tau(F_i^{(r)}) = E_i^{(r)}
\end{equation}
is an anti-automorphism of order 2 and its restriction on $Y_{1|n}(\sigma)$ 
gives an anti-isomorphism $\tau:Y_{1|n}(\sigma)\rightarrow Y_{1|n}(\sigma^t)$, 
where $\sigma^t$ is the transpose of the matrix $\sigma$.

Suppose instead that
$\vec\sigma = (\vec{s}_{i,j})_{1 \leq i,j \leq n+1}$ is another shift matrix satisfying (\ref{sijk}) and
in addition $\vec{s}_{i,i+1}+\vec s_{i+1,i}= s_{i,i+1}+s_{i+1,i}$ for all $i=1,\dots,n$.
Another check of relations shows that the map $\iota:Y_{1|n}(\sigma) \rightarrow Y_{1|n}({\vec{\sigma}})$ defined by
\begin{equation}\label{iotadef}
\iota(D_i^{(r)}) = \vec{D}_i^{(r)},\quad
\iota(E_i^{(r)}) = \vec{E}_i^{(r-s_{i,i+1}+\vec s_{i,i+1})},\quad
\iota(F_i^{(r)}) = \vec{F}_i^{(r-s_{i+1,i}+\vec s_{i+1,i})},
\end{equation}
is a superalgebra isomorphism. Here and later on we denote the generators $D_i^{(r)}, E_i^{(r)}$ and $F_i^{(r)}$ of $Y_{1|n}(\vec\sigma)$ instead by $\vec{D}_i^{(r)}, \vec{E}_i^{(r)}$ and $\vec{F}_i^{(r)}$ to avoid possible confusion.

\section{Parabolic presentations}\label{parabo}
In this section, we introduce the notion of parabolic presentation to the shifted super Yangian $Y_{1|n}(\sigma)$. We start with a brief review about parabolic presentations of $Y_{1|n}$ from \cite{Pe1} in \textsection 3.1, and then we extend the notion to $Y_{1|n}(\sigma)$ in \textsection 3.2.

\subsection{Parabolic presentations of $Y_{1|n}$}
Recall that $Y_{1|n}$ is generated by the elements $t_{i,j}^{(r)}$ with defining relation (\ref{Nadef}). We define the formal power series
\[t_{ij} (u) = \delta_{ij} + t_{ij}^{(1)} u^{-1} + t_{ij}^{(2)}u^{-2} + t_{ij}^{(3)}u^{-3}+\ldots
\] and the matrix $T(u):=\big(t_{ij}(u)\big)_{1\leq i,j\leq n+1}$.

Let $\nu$ be a composition of $n$ with length $m$. For notational reason, we set
\[
\mu_1=1 \qquad\text{and}\qquad \mu_{j}=\nu_{j-1} \quad\text{for all}\quad 2\leq j\leq m+1,
\]
and $\mu=(\mu_1\,|\, \mu_{2},\mu_{3},\ldots,\mu_{m+1})$ denotes the composition of ($1|n$).

By definition, the leading minors of the matrix $T(u)$ are invertible. Depending on the given composition $\mu$, $T(u)$ possesses a $Gauss$ $decomposition$
\begin{equation*}
T(u) = F(u) D(u) E(u)
\end{equation*}
for unique {\em block matrices} $D(u)$, $E(u)$ and $F(u)$ of the form
$$
D(u) = \left(
\begin{array}{cccc}
D_{1}(u) & 0&\cdots&0\\
0 & D_{2}(u) &\cdots&0\\
\vdots&\vdots&\ddots&\vdots\\
0&0 &\cdots&D_{m+1}(u)
\end{array}
\right),
$$

$$
E(u) =
\left(
\begin{array}{cccc}
I_{\mu_1} & E_{1,2}(u) &\cdots&E_{1,m+1}(u)\\
0 & I_{\mu_2} &\cdots&E_{2,m+1}(u)\\
\vdots&\vdots&\ddots&\vdots\\
0&0 &\cdots&I_{\mu_{m+1}}
\end{array}
\right),\:
$$

$$
F(u) = \left(
\begin{array}{cccc}
I_{\mu_1} & 0 &\cdots&0\\
F_{2,1}(u) & I_{\mu_2} &\cdots&0\\
\vdots&\vdots&\ddots&\vdots\\
F_{m+1,1}(u)&F_{m+1,2}(u) &\cdots&I_{\mu_{m+1}}
\end{array}
\right),
$$
where
\begin{align}
D_a(u) &=\big(D_{a;i,j}(u)\big)_{1 \leq i,j \leq \mu_a},\label{da}\\
E_{a,b}(u)&=\big(E_{a,b;i,j}(u)\big)_{1 \leq i \leq \mu_a, 1 \leq j \leq \mu_b},\label{eab}\\
F_{b,a}(u)&=\big(F_{b,a;i,j}(u)\big)_{1 \leq i \leq \mu_b, 1 \leq j \leq \mu_a},\label{fba}
\end{align}
are $\mu_a \times \mu_a$,
$\mu_a \times \mu_b$
and  $\mu_b \times\mu_a$ matrices, respectively, for all $1\le a\le m+1$ in (\ref{da})
and all $1\le a<b\le m+1$ in (\ref{eab}) and (\ref{fba}).

For all $1\leq a\leq m+1$, define the $\mu_a\times\mu_a$ matrix
$D_a^{\prime}(u)=\big(D_{a;i,j}^{\prime}(u)\big)_{1\leq i,j\leq \mu_a}$ by
\begin{equation*}
D_a^{\prime}(u):=\big(D_a(u)\big)^{-1}.
\end{equation*}
The entries of these matrices are expanded into power series
\begin{eqnarray*}
D_{a;i,j}(u) = \sum_{r \geq 0} D_{a;i,j}^{(r)} u^{-r}, 
E_{a,b;i,j}(u) = \sum_{r \geq 1} E_{a,b;i,j}^{(r)} u^{-r},\\
D_{a;i,j}^{\prime}(u) = \sum_{r \geq 0} D^{\prime(r)}_{a;i,j} u^{-r},
F_{b,a;i,j}(u) = \sum_{r \geq 1} F_{b,a;i,j}^{(r)} u^{-r}.
\end{eqnarray*}
Moreover, for $1\leq a\leq n$, we set
\begin{eqnarray*}
E_{a;i,j}(u) :=& E_{a,a+1;i,j}(u)=\sum_{r \geq 1} E_{a;i,j}^{(r)} u^{-r},\\
F_{a;i,j}(u) :=& F_{a+1,a;i,j}(u)=\sum_{r \geq 1} F_{a;i,j}^{(r)} u^{-r}.
\end{eqnarray*}

\begin{proposition}\cite[Theorem 4]{Pe1}\label{Per} The super Yangian $Y_{1|n}$ is generated by the elements \begin{align*}
&\lbrace D_{a;i,j}^{(r)}, D_{a;i,j}^{\prime(r)} \,|\, 1\leq a\leq n+1, 1\leq i,j\leq \mu_a, r\geq 0\rbrace,\\
&\lbrace E_{a;i,j}^{(r)} \,|\, 1\leq a \leq n, 1\leq i\leq \mu_a, 1\leq j\leq\mu_{a+1}, r\geq 1\rbrace,\\
&\lbrace F_{a;i,j}^{(r)} \,|\, 1\leq a \leq n, 1\leq i\leq\mu_{a+1}, 1\leq j\leq \mu_a, r\geq 1\rbrace,
\end{align*}
subject to certain relations depending on $\mu$.
\end{proposition}
Since the defining relations are exactly the relations (\ref{parashiftfirst})-(\ref{parashiftlast}) below when all $s_{i,j}=0$, we omit the relations here.

For example, the presentation of $Y_{1|n}$ introduced in \textsection 2  is the special case when $\mu=~(1^{n+1})$, where $D^{(r)}_{a;1,1}=D^{(r)}_a$, $E^{(r)}_{a;1,1}=E^{(r)}_a$ and $F^{(r)}_{a;1,1}=F^{(r)}_a$. The presentation depends on the composition $\mu$ and hence we will use the notation $Y_{\mu}=Y_{1|n}$ to emphasize the composition which we are using.

\subsection{Parabolic presentation of $Y_{1|n}(\sigma)$}
Throughout this subsection, we fix a shift matrix $\sigma=(s_{i,j})_{1\leq i,j\leq n+1}$. Also recall that we set $\mu_1=1$ for notational reason.
\begin{definition}
A composition $\mu=(\mu_1\,|\, \mu_2,\ldots,\mu_{m+1})$ of $(1|n)$ is said to be {\em admissible for} $\sigma$ if $s_{i,j}=0$ for all $\mu_1+\ldots+\mu_{a-1}+1\leq i,j\leq \mu_1+\ldots+\mu_a$, for each $1\leq a\leq m+1$.
\end{definition}
For example, the composition $\mu=(1^{n+1})$ is admissible for any shift matrix $\sigma$. We will use a shorthand notation from now on
\begin{equation}\label{sabmu}
s_{a,b}^{\mu}:=s_{\mu_1+\ldots+\mu_a,\mu_1+\ldots+\mu_b}.
\end{equation}
Note that one can recover the original matrix $\sigma$ if the admissible shape $\mu$ and the numbers $\{s_{a,b}^{\mu}|1\leq a,b\leq m+1\}$ are given. Assume from now on that a shift matrix $\sigma$ and an admissible shape $\mu$ for $\sigma$ are given.

\begin{definition}\label{parashift}
The shifted super Yangian of $\gl_{1|n}$ associated to $\sigma$ and $\mu$, 
denoted by $Y_{\mu}(\sigma)$, is the superalgebra 
over $\C$ generated by the following elements
\[\lbrace D_{a;i,j}^{(r)}| 1\leq a\leq m+1, 1\leq i,j\leq \mu_{a}, r\geq 0\rbrace,\]
\[\lbrace E_{a;i,j}^{(r)}| 1\leq a\leq m, 1\leq i\leq \mu_a, 1\leq j\leq \mu_{a+1} ,r> s_{a,a+1}^{\mu}\rbrace,\]
\[\lbrace F_{a;i,j}^{(r)}| 1\leq a\leq m, 1\leq i\leq \mu_{a+1 }, 1\leq j\leq \mu_{a} ,r> s_{a+1,a}^{\mu}\rbrace,\]
subject to the following defining relations:

\begin{eqnarray}\label{parashiftfirst}
D_{a;i,j}^{(0)}&=&\delta_{ij}\,,\\
\sum_{t=0}^{r}D_{a;i,p}^{(t)}D_{a;p,j}^{\prime (r-t)}&=&\delta_{r0}\delta_{ij}\,,\\
\big[D_{a;i,j}^{(r)},D_{b;h,k}^{(s)}\big]&=&
    \delta_{ab}(-1)^{\tp(a)}\sum_{t=0}^{min(r,s)-1}\big(D_{a;h,j}^{(t)}D_{a;i,k}^{(r+s-1-t)}-D_{a;h,j}^{(r+s-1-t)}D_{a;i,k}^{(t)}\big),\label{ddrel}
\end{eqnarray}
\begin{align}\notag
&\lefteqn{[D_{a;i,j}^{(r)}, E_{b;h,k}^{(s)}]=}\\
&\;\left\{
  \begin{array}{ll}
     \displaystyle   -\delta_{a,b}\delta_{h,j}\sum_{t=0}^{r-1} D_{a;i,p}^{(t)}E_{a;p,k}^{(r+s-1-t)}
         +\delta_{a,b+1}\sum_{t=0}^{r-1}D_{a;i,k}^{(t)}E_{b;h,j}^{(r+s-1-t)}, \;b \ne 1,\\[4mm]
     \displaystyle    \delta_{a,b}\delta_{h,j}\sum_{t=0}^{r-1} D_{a;i,p}^{(t)}E_{a;p,k}^{(r+s-1-t)}
          +\delta_{a,b+1}\sum_{t=0}^{r-1}D_{a;i,k}^{(t)}E_{b;h,j}^{(r+s-1-t)}, \;b = 1,
 \end{array}\right.\\[4mm]\notag
&\lefteqn{[D_{a;i,j}^{(r)}, F_{b;h,k}^{(s)}]=}\\
&\left\{
  \begin{array}{ll}
  \displaystyle \delta_{a,b}\delta_{k,i}\sum_{t=0}^{r-1}
  F_{b;h,p}^{(r+s-1-t)}D_{a;p,j}^{(t)}
  -\delta_{a,b+1}\sum_{t=0}^{r-1}F_{b;i,k}^{(r+s-1-t)}D_{a;h,j}^{(t)}, \;b \ne 1,\\[4mm]
   \displaystyle   -\delta_{a,b}\delta_{k,i}\sum_{t=0}^{r-1}F_{b;h,p}^{(r+s-1-t)}D_{a;p,j}^{(t)}
      -\delta_{a,b+1}\sum_{t=0}^{r-1}F_{b;i,k}^{(r+s-1-t)}D_{a;h,j}^{(t)}\,, \;b = 1,
\end{array}\right.
\end{align}
\begin{equation}
[E_{a;i,j}^{(r)}, E_{a;h,k}^{(s)}]=
 \left\{
  \begin{array}{ll}
    \displaystyle    -\sum_{t=1}^{s-1}E_{a;i,k}^{(t)}E_{a;h,j}^{(r+s-1-t)}
          +\sum_{t=1}^{r-1}E_{a;i,k}^{(t)}E_{a;h,j}^{(r+s-1-t)}\big),& a\neq 1,\\[4mm]
     \displaystyle     \sum_{t=1}^{r-1}E_{a;i,k}^{(t)}E_{a;h,j}^{(r+s-1-t)}
          -\sum_{t=1}^{s-1}E_{a;i,k}^{(t)}E_{a;h,j}^{(r+s-1-t)} ,& a=1,\\[4mm]
  \end{array}\right.\\
\end{equation}
\begin{equation}
[F_{a;i,j}^{(r)}, F_{a;h,k}^{(s)}]=
\left\{
  \begin{array}{ll}
   \displaystyle -\sum_{t=1}^{r-1}F_{a;i,k}^{(r+s-1-t)}F_{a;h,j}^{(t)}
   +\sum_{t=1}^{s-1}F_{a;i,k}^{(r+s-1-t)}F_{a;h,j}^{(t)}\big),& a\neq 1,\\
   \displaystyle \sum_{t=1}^{r-1}F_{a;i,k}^{(r+s-1-t)}F_{a;h,j}^{(t)}-
   \sum_{t=1}^{s-1}F_{a;i,k}^{(r+s-1-t)}F_{a;h,j}^{(t)} \,,& a=1,\\
  \end{array}\right.\\
\end{equation}
\begin{equation}\label{ef=dd}
[E_{a;i,j}^{(r)}, F_{b;h,k}^{(s)}]=
         \delta_{a,b} \sum_{t=0}^{r+s-1}D_{a+1;h,j}^{(r+s-1-t)} D^{\prime (t)}_{a;i,k}\,,
\end{equation}
\begin{align}  \label{eedeg}
&[E_{a;i,j}^{(r+1)}, E_{a+1;h,k}^{(s)}]-[E_{a;i,j}^{(r)}, E_{a+1;h,k}^{(s+1)}]
=-\delta_{h,j}E_{a;i,q}^{(r)}E_{a+1;q,k}^{(s)}\,,\\[3mm]\label{ffdeg}
&[F_{a;i,j}^{(r+1)}, F_{a+1;h,k}^{(s)}]-[F_{a;i,j}^{(r)}, F_{a+1;h,k}^{(s+1)}]
=\delta_{i,k}F_{a+1;h,q}^{(s)}F_{a;q,j}^{(r)}\,,\\[3mm]\label{ee0}
&[E_{a;i,j}^{(r)}, E_{b;h,k}^{(s)}] = 0
\quad\text{\;\;if\;\; $b>a+1$ \;\;or\;\; \;if\;\;$b=a+1$ and $h \neq j$},\\[3mm]\label{ff0}
&[F_{a;i,j}^{(r)}, F_{b;h,k}^{(s)}] = 0
\quad\text{\;\;if\;\; $b>a+1$ \;\;or\;\; \;if\;\;$b=a+1$ and $i \neq k$},
\end{align}
\begin{eqnarray}\label{eee}
&\big[E_{a;i,j}^{(r)},[E_{a;h,k}^{(s)},E_{b;f,g}^{(l)}]\big]+
\big[E_{a;i,j}^{(s)},[E_{a;h,k}^{(r)},E_{b;f,g}^{(l)}]\big]=0, &|a-b|\geq 1,\\[3mm]\label{fff}
&\big[F_{a;i,j}^{(r)},[F_{a;h,k}^{(s)},F_{b;f,g}^{(l)}]\big]+
\big[F_{a;i,j}^{(s)},[F_{a;h,k}^{(r)},F_{b;f,g}^{(l)}]\big]=0, &|a-b|\geq 1,\label{parashiftlast}
\end{eqnarray}
for all admissible $a,b,f,g,h,i,j,k,l,r,s,t$,
where the index $p$ (respectively, $q$) is summed over 1,$\ldots,\mu_a$ (respectively, $1,\ldots,\mu_{a+1}$).
Similarly, the only odd generators are \[\{E_{1;i,j}^{(r)}| 1\leq i\leq \mu_1, 1\leq j\leq \mu_2, r> s_{1,2}^\mu\}\cup \{F_{1;i,j}^{(r)}| 1\leq i\leq \mu_2, 1\leq j\leq \mu_1, r> s_{2,1}^\mu\}.\]
\end{definition}

Let $\Gamma$ denote the homomorphism $Y_{\mu}(\sigma)\longrightarrow Y_{\mu}$ sending the generators
$D^{(r)}_{a;i,j}$, $D^{\prime (r)}_{a;i,j}$, $E^{(r)}_{a;i,j}$ and $F^{(r)}_{a;i,j}$ in $Y_{\mu}(\sigma)$ to those in $Y_{\mu}$ with the same notations obtained by Gauss decomposition. We now prove that the map $\Gamma$ is injective and its image is independent of the choice of the admissible shape $\mu$. In particular, $Y_{\mu}(\sigma)$ can be identified with the super Yangian $Y_{1|n}(\sigma)$ introduced in \textsection 2 which is the special case when $\mu=(1^{n+1})$ of our current definition.

 For $1\leq a<b\leq m+1$, $1\leq i\leq \mu_a$, $1\leq j\leq \mu_b$, $r>s_{a,b}^{\mu}$ and a choice of $1\leq k\leq \mu_{b-1}$, we define the elements $E_{a,b;i,j}^{(r)}\in Y_{1|n}(\sigma)$ recursively by
\begin{equation}\label{eparag}
E_{a,a+1;i,j}^{(r)}:= E_{a;i,j}^{(r)}, \qquad E_{a,b;i,j}^{(r)}:=-[E_{a,b-1;i,k}^{(r-s_{b,b-1}^{\mu})}, E_{b-1;k,j}^{(s_{b-1,b}^{\mu}+1)}].
\end{equation}
Similarly, for $1\leq a<b\leq m+1$, $1\leq i\leq \mu_b$, $1\leq j\leq \mu_a$, $r>s_{b,a}^{\mu}$ and a choice of $1\leq k\leq \mu_{b-1}$, we define the elements $F_{b,a;i,j}^{(r)}\in Y_{1|n}(\sigma)$ by
\begin{equation}\label{fparag}
F_{a+1,a;i,j}^{(r)}:= F_{a;i,j}^{(r)}, \qquad F_{b,a;i,j}^{(r)}:=-[F_{b-1;i,k}^{(s_{b,b-1}^{\mu}+1)}, F_{b-1,a;k,j}^{(r-s_{b,b-1}^{\mu})}],
\end{equation}
It turns out that the above definitions are independent of the choice of $k$; see \cite[(6.9)]{BK1} for detail.

Similar to \textsection 2, we introduce the $loop$ $filtration$ on $Y_{\mu}(\sigma)$
\begin{equation}\notag
L_0 Y_{\mu}(\sigma) \subseteq L_1 Y_{\mu}(\sigma) \subseteq L_2 Y_{\mu}(\sigma) \subseteq \cdots
\end{equation}
by setting the degree of the generators $D_{a;i,j}^{(r)}$, $E_{a;i,j}^{(r)}$, and $F_{a;i,j}^{(r)}$ to be $(r-1)$ and $L_k Y_{\mu}(\sigma)$ to be the span of all supermonomials in the generators of total degree $\leq k$ and denote the associated graded algebra by $\gr^LY_{\mu}(\sigma)$.

Define the elements $\left\lbrace e_{i,j;r}\right\rbrace_{1\leq i,j\leq n+1, r\geq s_{i,j}}$ in $\gr^LY_{\mu}(\sigma)$ by
\begin{eqnarray}
e_{\mu_1+\cdots+\mu_{a-1}+i,\mu_1+\cdots+\mu_{a-1}+j;r}=\gr^L_{r} D_{a;i,j}^{(r+1)}\;,\\
e_{\mu_1+\cdots+\mu_{a-1}+i,\mu_1+\cdots+\mu_{b-1}+j;r}=\gr^L_{r} E_{a,b;i,j}^{(r+1)}\;,\\
e_{\mu_1+\cdots+\mu_{b-1}+i,\mu_1+\cdots+\mu_{a-1}+j;r}=\gr^L_{r} F_{b,a;i,j}^{(r+1)}\;.
\end{eqnarray}
One can prove that these elements satisfy relation (\ref{s-inj}) as well. As a result, there exists a surjective superalgebra homomorphism $\gamma:U\big(\gl_{1|n}[t](\sigma)\big)\twoheadrightarrow \gr^LY_{\mu}(\sigma)$ such that $\gamma(e_{i,j}t^r)=(-1)^{\tp(i)}e_{i,j;r}$. By the PBW theorem for $\gl_{1|n}[t](\sigma)$, $\gamma$ is also injective and hence an isomorphism.

Let $Y_{\mu}^0$ denote the subalgebra of $Y_{1|n}(\sigma)$ generated by all the $D_{a;i,j}^{(r)}$'s , $Y_{\mu}^+(\sigma)$ denote the subalgebra generated by all the $E_{a;i,j}^{(r)}$'s and $Y_{\mu}^-(\sigma)$ denote the subalgebra generated by all the $F_{a;i,j}^{(r)}$'s. The following corollaries are parallel to Corollary \ref{pbw1} and Corollary \ref{iso1}.


\begin{corollary}\label{pbw2}
(1) The set of monomials in the elements
$\{D_{a;i,j}^{(r)}\}_{a=1,\dots,m+1, 1 \leq i,j \leq \mu_a,r>0}$
taken in some
fixed order forms a basis for $Y_\mu^0$.\\
(2) The set of supermonomials in the elements
$\{E_{a,b;i,j}^{(r)}\}_{1 \leq a < b \leq m+1, 1 \leq i \leq \mu_a, 1 \leq j \leq \mu_b, r >s_{a,b}^{\mu}}$ taken in some fixed
order forms a basis for $Y_\mu^+(\sigma)$.\\
(3) The set of supermonomials in the elements
$\{F_{b,a;i,j}^{(r)}\}_{1 \leq a < b \leq m+1,
1 \leq i \leq \mu_b, 1 \leq j \leq \mu_a, r > s_{b,a}^{\mu}}$ taken in some fixed
order forms a basis for $Y_\mu^-(\sigma)$.\\
(4) The set of supermonomials in the union of
the elements listed in (1)--(3)
taken in some fixed order forms a basis for $Y_{\mu}(\sigma)$.
\end{corollary}

\begin{corollary}\label{iso2}
The multiplicative map  $Y_\mu^-(\sigma) \otimes Y_\mu^0 \otimes Y_\mu^+(\sigma)
\longrightarrow Y_{1|n}(\sigma)$ is an isomorphism of vector spaces.
\end{corollary}

Now we prove that the definition of $Y_{\mu}(\sigma)$ is independent 
of the choice of the admissible shape $\mu$. If $\mu_i=1$ for all $i$, 
then $Y_{\mu}(\sigma)$ is defined as in \textsection 2. Suppose that $\mu_b>1$ for some $2\leq b\leq m+1$. 
We may further assume that $\mu_b=\alpha+\beta$ for some $\alpha, \beta\geq 1$. 
Define a finer composition $\nu$ of $(1|n)$ by
\[
\nu=(\mu_1\,|\, \mu_2,\ldots,\mu_{b-1},\alpha,\beta,\mu_{b+1},\ldots,\mu_{m+1}).
\]
Note that $\nu$ is also an admissible shape for $\sigma$. As a result, the matrix $T(u)$ has a Gauss decomposition with respect to $\nu$ as well; that is,
\[
T(u)= {^\mu E(u)} {^\mu D(u)} {^\mu F(u)}= {^\nu E(u)} {^\nu D(u)} {^\nu F(u)},
\]
where the matrices are block matrices as introduced in the beginning of this section.

We will denote by $^\mu D_a$ and $^\nu D_a$ to be the $a$-th diagonal matrices in $^\mu D(u)$ and $^\nu D(u)$ with respect to the compositions $\mu$ and $\nu$, respectively. Similarly, $^\mu E_a$ and $^\mu F_a$ are defined to be the matrices in the $a$-th upper and lower-diagonal of $^\mu E(u)$ and $^\mu F(u)$, respectively; $^\nu E_a$ and $^\nu F_a$ are defined to be the matrices in the $a$-th upper and lower-diagonal of $^\nu E(u)$ and $^\nu F(u)$, respectively.

\begin{lemma}\label{split}
In the notation above, define an $(\alpha \times \alpha)$-matrix
$A$, an $(\alpha\times \beta)$-matrix $B$,
a $(\beta \times \alpha)$-matrix $C$ and a $(\beta \times \beta)$-matrix $D$
from the equation
$$
{^\mu}D_b = \left(\begin{array}{ll}I_\alpha&0\\
C&I_\beta\end{array}
\right)\left(\begin{array}{ll}A&0\\
0&D\end{array}
\right)\left(\begin{array}{ll}I_\alpha&B\\
0&I_\beta\end{array}
\right).
$$
Then,
\begin{itemize}
\item[(i)] ${^\nu}D_a = {^\mu}D_a$ for $a < b$, ${^\nu}D_b = A$,
${^\nu}D_{b+1} = D$, and ${^\nu}D_c = {^\mu}D_{c-1}$ for $c > b+1$;
\item[(ii)] ${^\nu}E_a = {^\mu}E_a$ for $a < b-1$,
${^\nu}E_{b-1}$ is the submatrix consisting
of the first $\alpha$ columns
of
${^\mu}E_{b-1}$, ${^\nu}E_{b} = B$, ${^\nu}E_{b+1}$ is the submatrix consisting of the last $\beta$ rows
of ${^\mu}E_b$, and ${^\nu}E_c = {^\mu}E_{c-1}$ for $c > b+1$;
\item[(iii)] ${^\nu}F_a = {^\mu}F_a$ for $a < b-1$,
${^\nu}F_{b-1}$ is the submatrix consisting of the first
$\alpha$ rows  of
${^\mu}F_{b-1}$, ${^\nu}F_{b} = C$, ${^\mu}F_{b+1}$ is the submatrix
consisting of the last $\beta$ columns of
${^\mu}F_b$, and ${^\nu}F_c = {^\mu}F_{c-1}$ for $c > b+1$;
\end{itemize}
\end{lemma}

\begin{proof}
Multiply matrices.
\end{proof}

To show the definition of $Y_\mu(\sigma)$ is independent the choice of $\mu$, it suffices to show that $Y_\nu(\sigma)=Y_\mu(\sigma)$. By Lemma \ref{split}, one has that $Y_{\nu}(\sigma)\subseteq Y_{\mu}(\sigma)$. Now the equality follows from the fact that the isomorphism $U\big(\gl_{1|n}[t](\sigma)\big)\cong \gr^L Y_{\mu}(\sigma)$ is independent of $\mu$.

\begin{proposition}\label{indmu}
The superalgebra $Y_{\mu}(\sigma)$ is independent of the choice of the admissible shape $\mu$.
\end{proposition}

Using Lemma \ref{split} and induction on the length of the composition, one can describe the maps $\tau$ and $\iota$ in terms of parabolic generators. The anti-isomorphism $\tau$ satisfies
\begin{equation}
\tau(D_{a;i,j}^{(r)}) =
D_{a;j,i}^{(r)},\,\,\,
\tau(E_{a;i,j}^{(r)}) =
F_{a;j,i}^{(r)},\,\,\,
\tau(F_{a;i,j}^{(r)}) =
E_{a;j,i}^{(r)}.
\end{equation}

Also, in terms of the notation of (\ref{iotadef}), suppose that $\mu$ is an admissible shape for $\sigma$ and $\vec\sigma$ simultaneously, then the isomorphism $\iota:Y_{\mu}(\sigma)\rightarrow Y_{\mu}(\vec\sigma)$ satisfies
\begin{multline}
\quad\iota(D_{a;i,j}^{(r)}) =
\vec D_{a;i,j}^{(r)},\,\,\,
\iota(E_{a;i,j}^{(r)}) =
\vec E_{a;i,j}^{(r-s_{a,a+1}^\mu +\vec s_{a,a+1}^\mu )},\,\,\,
\iota(F_{a;i,j}^{(r)}) =
\vec F_{a;i,j}^{(r-s_{a+1,a}^\mu +\vec s_{a+1,a}^\mu )}.
\label{iotadefpara}
\end{multline}

\section{Baby comultiplications}
In this section, we define some comultiplication-like maps on $Y_{\mu}(\sigma)$ that are crucial in later sections. 

We first consider the special case where $\sigma=0$, i.e., the whole super Yangian $Y_{1|n}$. It is well-known (cf. \cite{Go}) that $Y_{1|n}$ is a Hopf superalgebra where the comultiplication $\Delta:Y_{1|n}\rightarrow Y_{1|n}\otimes Y_{1|n}$ is defined in terms of the RTT generators $\{t_{i,j}^{(r)}\}$ by the formula
\begin{equation}
\Delta(t_{i,j}^{(r)})=\sum_{s=0}^r\sum_{k=1}^{n+1}t_{i,k}^{(s)}\otimes t_{k,j}^{(r-s)}.
\end{equation}
Moreover, the evaluation homomorphism $\ev:Y_{1|n}\rightarrow U(\gl_{1|n})$ defined by
\begin{equation}
\ev(t_{i,j}^{(r)})=\delta_{r,0}\delta_{i,j}+\delta_{r,1}(-1)^{\tp(i)}e_{i,j},
\end{equation}
where $e_{i,j}$ is the elementary matrices in $\gl_{1|n}$, is surjective.

Define $\Delta_{\tilde{R}}:=(\text{id}\otimes \ev)\circ\Delta$ and $\Delta_{\tilde{L}}:=(\ev\otimes \text{id})\circ\Delta$, where they are both algebra homomorphisms. Thus we have
\begin{align}
\Delta_{\tilde{R}}:Y_{1|n}\rightarrow Y_{1|n}\otimes U(\gl_{1|n}),\qquad t_{i,j}^{(r)}\mapsto t_{i,j}^{(r)}\otimes 1+\sum_{k=1}^{n+1}(-1)^{\tp(k)}t_{i,k}^{(r-1)}\otimes e_{k,j}\\
\Delta_{\tilde{L}}:Y_{1|n}\rightarrow U(\gl_{1|n})\otimes Y_{1|n},\qquad t_{i,j}^{(r)}\mapsto 1\otimes t_{i,j}^{(r)} +\sum_{k=1}^{n+1}(-1)^{\tp(i)}e_{i,k}\otimes t_{k,j}^{(r-1)}.
\end{align}

In general, if $\sigma$ is not the zero matrix, then $Y_\mu(\sigma)$ 
does not have a comultiplication. But we may define some nice homomorphisms called 
{\em baby comultiplications} as in \cite{BK2}, suggested from the maps $\Delta_{\tilde{R}}$ and $\Delta_{\tilde{L}}$ above.

\begin{definition}
An admissible composition $\mu$ of $(1\,|\,n)$ for $\sigma$ is called {\em minimal admissible} if the length of $\mu$ is minimal.
\end{definition}

For example, $\mu=(1\,|\,2,1,1)$ is a minimal admissible shape of $(1|4)$ for the shift matrix
\[
\sigma=
\left[
 \begin{array}{c|cccc}
  0 & 0 & 0 & 1 & 1 \\
  \hline  
  0 & 0 & 0 & 1 & 2 \\
  0 & 0 & 0 & 1 & 2 \\
  2 & 2 & 2 & 0 & 1\\
  4 & 4 & 4 & 2 & 0
\end{array}
\right].
\]
Note that the minimal admissible shape denotes the ``sizes" of 
largest zero square matrices in the diagonal of $\sigma$, {\em respecting the parity} (so we use (1,2) rather than (3) in the northwestern corner).
Throughout this section, we fix a shift matrix $\sigma$ and 
a minimal admissible shape $\mu$ of $(1|n)$ for $\sigma$.

For convenience, we let $\beta=\mu_{m+1}$. Since $\mu$ is a minimal admissible shape for $\sigma$, we have that $1\leq \beta\leq n$ and either $s_{n+1-\beta,n+2-\beta}\neq 0$ or $s_{n+2-\beta,n+1-\beta}\neq 0$.

\begin{theorem}\label{baby1}
Let $\mu=(\mu_1\,|\,\mu_2,\ldots,\mu_{m+1})$ be a minimal admissible shape of $(1|n)$ for the shift matrix $\sigma$. For $1\leq i,j\leq \beta$, let
 \[
 \tilde e_{i,j}:=e_{i,j}+\delta_{i,j}(n-1-\beta)\in U(\gl_\beta).
 \]
\begin{itemize}
\item[(\it{1})] Suppose that $s_{n+1-\beta,n+2-\beta}\neq 0$. Define $\dot\sigma = (\dot s_{i,j})_{1 \leq i,j \leq n+1}$ by
     \begin{equation}\label{babyr1}
     \dot s_{i,j} = \left\{
     \begin{array}{ll}
      s_{i,j}-1&\hbox{if\, $i \leq n+1-\beta < j$,}\\
      s_{i,j}&\hbox{otherwise.}
      \end{array}\right.
      \end{equation}
     Then the map $\Delta_{R}:Y_{1|n}(\sigma) \rightarrow Y_{1|n}(\dot\sigma) \otimes U(\mathfrak{gl}_{\beta})$
     defined by
     \begin{align*}
       D_{a;i,j}^{(r)} &\mapsto \dot D_{a;i,j}^{(r)} \otimes 1
       - \delta_{a,m+1} \sum_{k=1}^{\beta} \dot D_{a;i,k}^{(r-1)} \otimes \tilde e_{k,j},\\
       E_{a;i,j}^{(r)} &\mapsto \dot E_{a;i,j}^{(r)} \otimes 1
       - \delta_{a,m} \sum_{k=1}^{\beta}\dot E_{a;i,k}^{(r-1)} \otimes \tilde e_{k,j},\\
       F_{a;i,j}^{(r)} & \mapsto \dot F_{a;i,j}^{(r)} \otimes 1,
      \end{align*}
       is a superalgebra homomorphism.\\[4mm]
\item[(\it{2})] Suppose that $s_{n+2-\beta,n+1-\beta}\neq 0$. Define $\dot\sigma = (\dot s_{i,j})_{1 \leq i,j \leq n+1}$ by
     \begin{equation}\label{babyl1}
     \dot s_{i,j} = \left\{
     \begin{array}{ll}
      s_{i,j}-1&\hbox{if\, $j \leq n+1-\beta < i$,}\\
      s_{i,j}&\hbox{otherwise.}
      \end{array}\right.
      \end{equation}
       Then the map $\Delta_{L}:Y_{1|n}(\sigma) \rightarrow U(\mathfrak{gl}_{\beta})\otimes Y_{1|n}(\dot\sigma)$
       defined by 
       \begin{align*}
       D_{a;i,j}^{(r)} &\mapsto 1\otimes\dot D_{a;i,j}^{(r)}
       - \delta_{a,m+1} \sum_{k=1}^{\beta} \tilde e_{i,k}\otimes \dot D_{a;k,j}^{(r-1)},\\
       E_{a;i,j}^{(r)} &\mapsto 1\otimes \dot E_{a;i,j}^{(r)} ,\\
       F_{a;i,j}^{(r)} & \mapsto 1\otimes \dot F_{a;i,j}^{(r)}-
       \delta_{a,m} \sum_{k=1}^{\beta} \tilde e_{i,k}\otimes \dot F_{a;k,j}^{(r-1)},
      \end{align*}
       is a superalgebra homomorphism.
       \end{itemize}
\end{theorem}
\begin{proof}
Check that $\Delta_{R}$ and $\Delta_{L}$ preserve the defining relations in Definition \ref{parashift}. Similar to \cite[Theorem 4.2]{BK2}, to check (\ref{eee}) and(\ref{fff}), one needs to use (\ref{eedeg}), (\ref{ffdeg}), (\ref{ee0}) and (\ref{ff0}) multiple times.
\end{proof}

\begin{remark}
To avoid possible confusion, in the above theorem and hereafter, the elements $D_{a;i,j}^{(r)}$, $E_{a;i,j}^{(r)}$, and $F_{a;i,j}^{(r)}$ in $Y_{1|n}(\dot{\sigma})$ are denoted by $\dot D_{a;i,j}^{(r)}$, $\dot E_{a;i,j}^{(r)}$, and $\dot F_{a;i,j}^{(r)}$, where $\dot{\sigma}$ is defined by either (\ref{babyr1}) or (\ref{babyl1}), with respect to the same composition $\mu$. Note that $\mu$ is admissible but may no longer be minimal with respect to $\dot{\sigma}$.
\end{remark}

The next lemma shows how to compute the baby comultiplications on the general parabolic elements $E_{a,b;i,j}^{(r)}$ and $F_{b,a;i,j}^{(r)}$ under the same minimal admissible shape.

\begin{lemma}\label{lema4.4}
\begin{itemize}
\item[(\it{1})] Under the same assumption of Theorem \ref{baby1}(1), for all admissible $i,j,r$ and $1\leq a< b-1<m+1$, we have
\begin{align*}
\Delta_{R}(F_{b,a;i,j}^{(r)}) &= \dot F_{b,a;i,j}^{(r)} \otimes 1,\\
\Delta_{R}(E_{a,b;i,j}^{(r)}) &= \!\dot E_{a,b;i,j}^{(r)} \otimes 1,\qquad \hbox{if $b < m+1$},
\end{align*} 
and
\begin{multline*}
\Delta_{R}(E_{a,m+1;i,j}^{(r)}) = -\big([\dot E_{a,m;i,h}^{(r-s_{m,m+1}^\mu)},\dot E_{m;h,j}^{(s_{m,m+1}^\mu +1)}]\otimes 1 \big) 
-\sum_{k=1}^{\beta}\dot E_{a,m+1;i,k}^{(r-1)} \otimes \tilde e_{k,j},
\end{multline*}
for any $1 \leq h \leq \mu_{m}$.
\item[(\it{2})] Under the same assumption of Theorem \ref{baby1}(2), for all admissible $i,j,r$ and $1\leq a< b-1<m+1$, we have
\begin{align*}
\Delta_{L}(E_{a,b;i,j}^{(r)}) &= 1 \otimes \dot E_{a,b;i,j}^{(r)} ,\\
\Delta_{L}(F_{b,a;i,j}^{(r)}) &= 1 \otimes \!\dot F_{b,a;i,j}^{(r)},\qquad \hbox{if $b < m+1$},
\end{align*} 
and
\begin{multline*}
\Delta_{L}(F_{m+1,a;i,j}^{(r)}) = -
\big( 1 \otimes [\dot F_{m;i,h}^{(s_{m+1,m}^\mu+1)},\dot F_{m,a;h,j}^{(r-s_{m+1,m}^\mu)}] \big)
-\sum_{k=1}^{\beta} \tilde e_{k,j} \otimes \dot F_{m+1,a;k,j}^{(r-1)},
\end{multline*}
for any $1 \leq h \leq \mu_{m}$.
\end{itemize}
\end{lemma}

\begin{proof}
We compute $\Delta_{R}(E_{a,m+1;i,j}^{(r)})$ for $1\leq a\leq m$ in detail here, while others are similar. By definition, \[E_{a,m+1;i,j}^{(r)}=-
[E_{a,m;i,h}^{(r-s_{m,m+1}^\mu)},E_{m;h,j}^{(s_{m,m+1}^\mu+1)}],\]
for any $1\leq h\leq\mu_{m}$. 
Also, $\Delta_{R}(E_{a,m;i,h}^{(r-s_{m,m+1}^\mu)})= 
\dot E_{a,m;i,h}^{(r-s_{m,m+1}^\mu)} \otimes 1$.
Hence
\begin{align*}
\Delta_{R}(E_{a,m+1;i,j}^{(r)})&=
-\left[\dot E_{a,m;i,h}^{(r-s_{m,m+1}^\mu)} \otimes 1,
\dot E_{m;h,j}^{(s_{m,m+1}^\mu+1)} \otimes 1 \right]\\
&\qquad +\left[ \dot E_{a,m;i,h}^{(r-s_{m,m+1}^\mu)} \otimes 1, 
\sum_{k=1}^{\beta}\dot E_{m;h,k}^{(s_{m,m+1}^\mu)} \otimes \tilde e_{k,j}\right]\\
&=
-\left[\dot E_{a,m;i,h}^{(r-s_{m,m+1}^\mu)},
\dot E_{m;h,j}^{(s_{m,m+1}^\mu+1)}\right ] \otimes 1 -\sum_{k=1}^{\beta}\dot E_{a,m+1;i,k}^{(r-1)} \otimes \tilde e_{k,j},
\end{align*}
as claimed.
\end{proof}

\begin{proposition}\label{babyinj}
The maps $\Delta_{R}$ and $\Delta_{L}$ are injective, whenever they are defined. 
\end{proposition}
\begin{proof}
Let $\beta$ be defined as in Theorem \ref{baby1} and let $\epsilon:U(\mathfrak{gl}_\beta)\rightarrow \C$ be the homomorphism such that $\epsilon(\tilde e_{i,j})=0$ for $1 \leq i,j \leq \beta$. By definition, $Y_{1|n}(\sigma)$ and $Y_{1|n}(\dot \sigma)$ are subalgebras of $Y_{1|n}$ with $Y_{1|n}(\sigma) \subseteq Y_{1|n}(\dot\sigma)$. One may check that the compositions $m\circ (\id \otimes \epsilon) \circ \Delta_{R}$ and $m\circ(\epsilon \otimes \id)\circ \Delta_{L}$ coincide with the natural embedding $Y_{1|n}(\sigma) \hookrightarrow Y_{1|n}(\dot\sigma)$, where $m(a\otimes b):=ab$ (the usual multiplication).
\end{proof}

\section{Canonical filtration}
Recall the canonical filtration of $Y_{1|n}$
\[
F_0Y_{1|n}\subseteq F_1Y_{1|n}\subset F_2Y_{1|n}\subseteq \cdots
\]
defined by deg $t_{ij}^{(r)}:=r$, i.e., $F_dY_{1|n}$ is the span of all supermonomials in the generators $t_{ij}^{(r)}$ of total degree $\leq$ d. It is clear form (\ref{Nadef}) that the associated graded superalgebra $\gr Y_{1|n}$ is supercommutative.

Now we describe the canonical filtration by parabolic presentation. Let $\mu=(\mu_1\,|\ldots,\mu_{m+1})$ be a composition of $(1|n)$. By \cite[Proposition 3.1]{Pe1}, the parabolic generators $D_{a;i,j}^{(r)}$ $E_{a,b;i,j}^{(r)}$ and $F_{b,a;i,j}^{(r)}$ of $Y_{\mu}=Y_{1|n}$ are linear combinations of supermonomials in $t_{i,j}^{(s)}$ of total degree $r$.

On the other hand, if we set $D_{a;i,j}^{(r)}$, $E_{a,b;i,j}^{(r)}$ and $F_{b,a;i,j}^{(r)}$ all to be of degree $r$, by multiplying out the matrix equation $T(u)=F(u)D(u)E(u)$, each $t_{ij}^{(r)}$ is a linear combination of supermonomials in the parabolic generators of total degree $r$.
Hence we may describe $F_dY_{1|n}$ as the span of all supermonomials in the parabolic generators $D_{a;i,j}^{(r)}$ $E_{a,b;i,j}^{(r)}$ and $F_{b,a;i,j}^{(r)}$ of total degree $\leq d$.

For $1\leq a,b\leq n+1$, $1\leq i\leq \mu_a$, $1\leq j\leq \mu_b$ and $r>0$, 
define the following elements in $\gr Y_{1|n}$ by
\begin{equation}\label{canorel}
e_{a,b;i,j}^{(r)} := \left\{
\begin{array}{ll}
\gr_r D_{a;i,j}^{(r)}&\hbox{if $a=b$,}\\
\gr_r E_{a,b;i,j}^{(r)}&\hbox{if $a < b$,}\\
\gr_r F_{a,b;i,j}^{(r)}&\hbox{if $a > b$.}
\end{array}
\right.
\end{equation}
Note that the notation depends on the shape $\mu$ implicitly. Since $\gr Y_{1|n}$ is supercommutative, together with Corollary \ref{pbw2} ($4$) in the case $\sigma=0$, we have the following:

\begin{proposition}\label{canopbw}
For any shape $\mu=(\mu_1\,|\,\ldots,\mu_{m+1})$, $\gr Y_{1|n}$ is the free supercommutative superalgebra on generators $\lbrace e_{a,b;i,j}^{(r)}\rbrace_{1\leq a,b\leq {m+1}, 1\leq i\leq \mu_a, 1\leq j\leq \mu_{b}, r>0}$.
\end{proposition}

Next we show that the shifted super Yangian adapts the canonical filtration as well. Let $\mu=(\mu_1\,|\ldots,\mu_{m+1})$  be a given admissible shape for a fixed shift matrix $\sigma$. Here we use the notations $Y_{\mu}$ for $Y_{1|n}$ and $Y_{\mu}(\sigma)$ for $Y_{1|n}(\sigma)$ to emphasize the shape $\mu$. 

Since $Y_{\mu}(\sigma)$ is a subalgebra of $Y_{\mu}$, we may induce the canonical filtration of $Y_{\mu}$ to $Y_{\mu}(\sigma)$ by defining $F_dY_{\mu}(\sigma):=F_dY_{\mu}\cap Y_{\mu}(\sigma)$. Therefore, the inclusion map $Y_{\mu}(\sigma)\hookrightarrow Y_{\mu}$ is a filtered map and the induced map $\gr Y_{\mu}(\sigma)\rightarrow \gr Y_{\mu}$ is injective as well. Hence we may identify $\gr Y_{\mu}(\sigma)$ with a subalgebra of the supercommutative superalgebra $\gr Y_{\mu}$. The next theorem describes this subalgebra explicitly.

\begin{theorem}\label{scanopbw}
For an admissible shape $\mu=(\mu_1\,|\ldots,\mu_{m+1})$, $\gr Y_{\mu}(\sigma)$ is the subalgebra of $\gr Y_{\mu}$ generated by the elements $\lbrace e_{a,b;i,j}^{(r)}\rbrace_{1\leq a,b\leq m+1, 1\leq i\leq \mu_a, 1\leq j\leq \mu_b, r>s_{a,b}^\mu}$.
\end{theorem}
\begin{proof}
By relations (\ref{eedeg}) and (\ref{ffdeg}), the element $e_{a,b;i,j}^{(r)}$ of $\gr Y_{\mu}(\sigma)$ is identified with the element of the same notation in $\gr Y_{\mu}$. The theorem follows from Corollary \ref{pbw2} ($4$) and Proposition \ref{canopbw}.
\end{proof}

\begin{remark}\label{parascano}
One consequence of Theorem \ref{scanopbw} is that we may define the canonical filtration on $Y_{\mu}(\sigma)$ intrinsically by setting the elements $D_{a;i,j}^{(r)}$, $E_{a,b;i,j}^{(r)}$ and $F_{a,b;i,j}^{(r)}$ of $Y_{\mu}(\sigma)$ all to be degree $r$. This definition is independent of the choice of admissible shape $\mu$ by Proposition \ref{indmu}.
\end{remark}

\begin{remark}\label{injremark}
By definition, the comultiplication $\Delta:Y_{1|n}\rightarrow Y_{1|n}\otimes Y_{1|n}$ is a filtered map with respect to the canonical filtration. If we extend the canonical filtration of $Y_{1|n}(\dot{\sigma})$ to $Y_{1|n}(\dot{\sigma})\otimes U(\gl_{\beta})$ by declaring the degree of the matrix unit $e_{ij}\in\gl_{\beta}$ to be 1, then the baby comultiplications $\Delta_R$ and $\Delta_L$ in Theorem \ref{baby1} are filtered maps as well, provided that they are defined. Moreover, following the argument in Proposition \ref{babyinj}, we have that the associated graded maps $\gr\Delta_L$ and $\gr\Delta_R$ are injective.
\end{remark}

\section{Truncation}

Fix a shift matrix $\sigma=(s_{i,j})_{1\leq i,j \leq n+1}$. Choose an integer $\ell\geq s_{1,n+1}+s_{n+1,1}$, which will be called ``level'' later. For each $1\leq i\leq n+1$, set
\begin{equation}\label{defpi}
p_i:=\ell-s_{i,n+1}-s_{n+1,i}.
\end{equation}
This defines a tuple $(p_1,\ldots,p_{n+1})$ of integers such that $0\leq p_1\leq \cdots\leq p_{n+1}=\ell$.

Let $\mu=(\mu_1\,|\ldots,\mu_{m+1})$ be an admissible shape for $\sigma$. For each $1\leq a\leq m+1$, set
\begin{equation}\label{pamu}
 p_a^\mu:=p_{\mu_1+\ldots+\mu_a}.
\end{equation}
By definition, for each $1\leq a\leq m+1$,
we have $p_i=p_a^\mu$ for all $\mu_1+\cdots+\mu_{a-1}+1\leq i\leq \mu_1+\cdots+\mu_a$.

\begin{definition}
The $truncated$ $shifted$ $super$ $Yangian$ $of$ $level$ $\ell$, denoted by $Y_{1|n}^\ell(\sigma)$, is the quotient of $Y_{1|n}(\sigma)$ by the two-side ideal generated by the elements $\lbrace D_{1}^{(r)}\rbrace_{r>p_1}$.
\end{definition}

When $\sigma=0$, since $t_{11}^{(r)}=D_1^{(r)}$ for all $r$ and $p_1=\ell$, the above definition is exactly the super analogy of {\em Yangian of level $\ell$} due to Cherednik or the {\em truncated Yangian} in \cite{BK1}.

It should be clear from the context that we are dealing with $Y_{1|n}(\sigma)$ or $Y_{1|n}^\ell(\sigma)$ and hence we will use the same symbols $D_{a;i,j}^{(r)}$, $E_{a,b;i,j}^{(r)}$ and $F_{a,b;i,j}^{(r)}$ to denote the generators of $Y_{1|n}(\sigma)$ and their images in the quotient $Y_{1|n}^\ell(\sigma)$, by abusing notations.

It is clear that the anti-isomorphism $\tau$ defined in (\ref{taudef}) factors through the quotient and induces an anti-isomorphism
\begin{equation}
\tau:Y_{1|n}^\ell(\sigma)\rightarrow Y_{1|n}^\ell(\sigma^t).
\end{equation}
Similarly, let $\vec{\sigma}$ be another shift matrix satisfying that $\vec{s}_{i,i+1}+\vec{s}_{i+1,i}=s_{i,i+1}+s_{i+1,i}$ for all $1\leq i\leq n+1$. Then the isomorphism $\iota$ defined by (\ref{iotadef}) also induces an isomorphism
\begin{equation}\label{iotaiso}
\iota:Y_{1|n}^\ell(\sigma)\rightarrow Y_{1|n}^\ell(\vec{\sigma}).
\end{equation}

Consider the canonical filtration defined in \textsection 5. We obtain a filtration
\[
F_0Y_{1|n}^\ell(\sigma)\subseteq F_1Y_{1|n}^\ell(\sigma)\subseteq\cdots
\]
induced from the quotient map $Y_{1|n}(\sigma)\rightarrow Y_{1|n}^\ell(\sigma)$. By Remark \ref{parascano}, we may define directly that for a given admissible shape $\mu$ for $\sigma$, the elements $D_{a;i,j}^{(r)}$, $E_{a,b;i,j}^{(r)}$ and $F_{a,b;i,j}^{(r)}$ of $Y_{1|n}^\ell(\sigma)$ are all of degree $r$ and $F_dY_{1|n}^\ell(\sigma)$ is the span of all supermonomials in these elements of total degree $\leq d$.

For $1\leq a,b\leq m+1$, $1\leq i\leq \mu_a$, $1\leq j\leq \mu_b$ and $r>s_{a,b}^\mu$, define element $e_{a,b;i,j}^{(r)}$, by abusing notations again,  in the associative graded superalgebra $\gr Y_{1|n}^\ell(\sigma)$ according to exactly the same formula (\ref{canorel}), except that those $D$'s, $E$'s and $F$'s are now in the quotient. Since $\gr Y_{1|n}^\ell(\sigma)$ is a quotient of $\gr Y_{1|n}(\sigma)$, by Proposition~\ref{canopbw} and Theorem~\ref{scanopbw}, it is also supercommutative and is generated by the elements $\lbrace e_{a,b;i,j}^{(r)}\rbrace_{1\leq a,b\leq m+1, 1\leq i\leq\mu_a, 1\leq j\leq \mu_b, r>s_{a,b}^\mu}$. In fact, we only need  a finite set as a generating set of $\gr Y_{1|n}^\ell(\sigma)$, as suggested by the following lemma.

\begin{lemma}\label{lema6.2}
For any admissible shape $\mu=(\mu_1\,|\,\ldots,\mu_{m+1})$, $\gr Y_{1|n}^\ell(\sigma)$ is generated only by the elements $\lbrace e_{a,b;i,j}^{(r)} \rbrace_{1\leq a,b\leq m+1, 1\leq i\leq\mu_a, 1\leq j\leq \mu_b, s_{a,b}^\mu<r\leq s_{a,b}^\mu+p_{min(a,b)}^\mu}$.
\end{lemma}
\begin{proof}
It can be proved by using the same argument as in \cite[Lemma 6.1]{BK2}, except that the induction argument starts from $\mu=(1\,|\,n)$ and our notations are slightly different.

\end{proof}

Suppose that the shift matrix $\sigma$ is nonzero and the level $\ell>0$. Let $\beta$ be the size of the maximal zero square matrix in the southeastern corner of $\sigma$. Hence we have $1\leq \beta\leq n$ and either $s_{n+1-\beta,n+2-\beta}\neq 0$ or $s_{n+2-\beta,n+1-\beta}\neq 0$.

If $s_{n+1-\beta,n+2-\beta}\neq 0$, then one may easily check that the baby comultiplication $\Delta_R$ defined in Theorem \ref{baby1} factors through the quotient to induce a map 
\begin{equation}\label{babylr}
\Delta_R:Y_{1|n}^\ell(\sigma)\rightarrow Y_{1|n}^{\ell-1}(\dot{\sigma})\otimes U(\gl_\beta),
\end{equation}
where $\dot{\sigma}$ is defined by (\ref{babyr1}). Similarly, if $s_{n+2-\beta,n+1-\beta}\neq 0$, then the baby comultiplication $\Delta_L$ defined in Theorem \ref{baby1} induces a map 
 \begin{equation}\label{babyll}
\Delta_L:Y_{1|n}^\ell(\sigma)\rightarrow U(\gl_\beta)\otimes Y_{1|n}^{\ell-1}(\dot{\sigma}),
\end{equation}
where $\dot{\sigma}$ is defined by (\ref{babyl1}). By Remark \ref{injremark}, $\Delta_R$ and $\Delta_L$ are filtered maps so they induce the following homomorphisms of graded superalgebras
\begin{align}\label{grdelr}
\gr \Delta_R: \gr Y_{1|n}^\ell(\sigma)\rightarrow \gr\big( Y_{1|n}^{\ell-1}(\dot{\sigma})\otimes U(\gl_\beta)\big),\\ \label{grdell}
\gr \Delta_L: \gr Y_{1|n}^\ell(\sigma)\rightarrow \gr\big(U(\gl_\beta)\otimes Y_{1|n}^{\ell-1}(\dot{\sigma})\big).
\end{align}

\begin{theorem}\label{PBWSYLpara}
For any admissible shape $\mu=(\mu_1\,|\,\ldots,\mu_{m+1})$, $\gr Y_{1|n}^\ell(\sigma)$ is the free supercommutative superalgebra on generators 
\[
\lbrace e_{a,b;i,j}^{(r)} \rbrace_{1\leq a,b\leq m+1, 1\leq i\leq\mu_a, 1\leq j\leq \mu_b, s_{a,b}^\mu<r\leq s_{a,b}^\mu+p_{min(a,b)}^\mu}.
\]
Also, the maps $\gr\Delta_{R}$ and $\gr\Delta_{L}$ in (\ref{grdelr}) and (\ref{grdell}) are injective whenever they are defined, and so are the maps $\Delta_R$ and $\Delta_L$ in (\ref{babylr}) and (\ref{babyll}).
\end{theorem}
\begin{proof}
We prove by induction on $\ell$, where the case $\ell=0$ is trivial. Assume $\ell>0$ and the first statement holds for any smaller $\ell$. It suffices to prove the induction step in the special case where $\mu=(\mu_1\,|\ldots,\mu_{m+1})$ is the minimal admissible shape for $\sigma$. Thus, at least one of $\Delta_R$ or $\Delta_L$ is defined. We assume without loss of generality that $\Delta_R$ is defined, where the other case can be derived from this by applying the anti-isomorphism $\tau$ defined in (\ref{taudef}). Let $\dot{\sigma}$ be defined by (\ref{babyr1}).

Define the following elements in $\gr\big(Y_{1|n}^{\ell-1}(\dot\sigma)\otimes U(\gl_\beta)\big)$:
\begin{equation*}
\dot e_{a,b;i,j}^{(r)} := \left\{
\begin{array}{ll}
\gr_r \dot D_{a;i,j}^{(r)}\otimes 1&\hbox{if $a=b$,}\\[3mm]
\gr_r \dot E_{a,b;i,j}^{(r)}\otimes 1&\hbox{if $a < b$,}\\[3mm]
\gr_r \dot F_{a,b;i,j}^{(r)}\otimes 1&\hbox{if $a > b$,}
\end{array}\right.
\end{equation*}
and $x_{i,j}:=\gr_1 1\otimes e_{ij}$ for all $1\leq i,j\leq \beta$.

By Theorem \ref{baby1} ({\it 1}), Lemma \ref{lema4.4} ({\it 1}) and Lemma \ref{lema6.2}, there exists some polynomials $f_{a;i,j}^{(r)}$ all in the variables $\dot e_{a,b;i,j}^{(s)}$ such that 
\begin{equation}\label{6.3tec1}
\gr\Delta_R(e_{a,b;i,j}^{(r)})=\dot e_{a,b;i,j}^{(r)},
\end{equation}
for all $1\leq a\leq m+1$, $1\leq b\leq m$, $1\leq i\leq \mu_a$, $1\leq j\leq \mu_b$, $s_{a,b}^\mu<r\leq s_{a,b}^\mu+ p_{min(a,b)}^\mu$, and
\begin{equation}\label{6.3tec2}
\gr\Delta_R(e_{a,m+1;i,j}^{(r)})=\sum_{k=1}^{\beta}\dot e_{a,m+1;i,k}^{(r-1)}\otimes x_{k,j}+f_{a;i,j}^{(r)},
\end{equation}
for all $1\leq a\leq m+1$, $1\leq i\leq \mu_a$, $1\leq j\leq \mu_{m+1}$, $s_{a,m+1}^\mu<r\leq s_{a,m+1}^\mu+ p_a^\mu$, where $\dot e_{m+1,m+1;i,j}^{(0)}:=\delta_{ij}$. 

By the induction hypothesis and the PBW theorem for $U(\gl_\beta)$, the elements 
\begin{align*}
&\{x_{i,j}\}_{1\leq i,j\leq \beta},\\
&\{\dot e_{a,b;i,j}^{(r)}\}_{1\leq a\leq m+1, \,1\leq b\leq m, \,1\leq i\leq \mu_a, \,1\leq j\leq \mu_b, \,s_{a,b}^\mu<r\leq s_{a,b}^\mu+p_{min(a,b)}^\mu},\\
&\{\dot e_{a,m+1;i,j}^{(r)}\}_{1\leq a\leq m+1, \,1\leq i\leq \mu_a, \,1\leq j\leq \mu_{m+1}, \,s_{a,m+1}^\mu+\delta_{a,m+1}-1<r\leq s_{a,m+1}^\mu+p_a^\mu-1}
\end{align*}
are algebraically independent in $\gr\big(Y_{1|n}^{\ell-1}(\dot{\sigma})\otimes U(\gl_\beta) \big)$. Let $B$ denote the set consisting of the above elements.

By $(\ref{6.3tec1})$ and $(\ref{6.3tec2})$, we may express the images of the generators of $\gr Y_{1|n}^\ell(\sigma)$ from Lemma~ \ref{lema6.2} under the map $\gr\Delta_R$ in terms of the elements of $B$. As a result, the images are algebraically independent in $\gr \big(Y_{1|n}^{\ell-1}(\dot\sigma)\otimes U(\gl_\beta)\big)$ and hence $\gr\Delta_R$ is injective. This completes the proof of the induction step.
\end{proof}

As a corollary, we obtain a PBW theorem for $Y_{1|n}^\ell(\sigma)$ in terms of parabolic generators.
\begin{corollary}\label{pbwbasis}
For any admissible shape $\mu=(\mu_1\,|\ldots,\mu_{m+1})$, the set of supermonomials in the elements
\[\lbrace D_{a;i,j}^{(r)}| 1\leq a\leq m+1, 1\leq i,j\leq \mu_{a}, 0<r\leq p_a^\mu\rbrace,\]
\[\lbrace E_{a,b;i,j}^{(r)}| 1\leq a<b\leq m+1, 1\leq i\leq \mu_a, 1\leq j\leq \mu_b , s_{a,b}^\mu<r\leq s_{a,b}^{\mu}+p_a^\mu\rbrace,\]
\[\lbrace F_{b,a;i,j}^{(r)}| 1\leq a<b\leq m+1, 1\leq i\leq \mu_b, 1\leq j\leq \mu_{a} , s_{b,a}^\mu<r\leq s_{b,a}^{\mu}+p_a^\mu\rbrace,\]
taken in some fixed order forms a basis for $Y_{1|n}^\ell(\sigma)$.
\end{corollary}

\section{Finite $W$-superalgebras and pyramids}

Throughout this section, $\mathfrak{g}=\glMN$ and $(\, \cdot \, ,\,\cdot\,)$ denotes the non-degenerate even supersymmetric invariant bilinear form on $\g$ defined by $(x,y):=str(xy)$ for all $x,y \in \g$, where $str$ means the super trace form and $xy$ stands for the usual composition. Every element of $\g$ in our description is considered $\mathbb{Z}_2$-homogeneous unless mentioned specifically.

In \textsection 7.1, we recall the definition of a finite $W$-superalgebra, which is determined by a nilpotent element $e$ and a semisimple element $h$ of $\glMN$. In \textsection 7.2, a combinatorial object called {\em pyramid} is introduced so that we may encode $e$ and $h$ simultaneously by a diagram $\pi$. Finally, in \textsection 7.3, when $\pi$ satisfies certain restriction, we explain how to associate such a $\pi$ to $Y_{1|n}^\ell(\sigma)$, a shifted Yangian of level $\ell$.

\subsection{Finite $W$-superalgebras of $\glMN$}

Let $e$ be an even nilpotent element in $\g$. It can be shown (cf. \cite{Ho}, \cite{Wa}) that there exists (not uniquely) a semisimple element $h\in\g$ such that $\ad h:\g\rightarrow\g$ gives a {\em good $\mathbb{Z}$-grading of $\g$ for $e$}, which means the following:
\begin{enumerate}
\item[(1)] $\ad h(e)=2e$,
\item[(2)] $\g=\bigoplus_{j\in\mathbb{Z}} \g(j)$, where $\g(j):=\{x\in\g|\ad h(x)=jx\}$,
\item[(3)] the center of $\g$ is contained in $\g(0)$,
\item[(4)] $\ad e:\g(j)\rightarrow\g(j+2)$ is injective for all $j\leq -1$,
\item[(5)] $\ad e:\g(j)\rightarrow\g(j+2)$ is surjective for all $j\geq -1$.
\end{enumerate}

\begin{example}
Let $e$ be an even nilpotent element. By Jacobson-Morozov Theorem, there exist $h\in\g$ and $f\in\g$ such that $\{e,h,f\}$ forms an $\mathfrak{sl}_2$-triple. It follows from the representation theory of $\mathfrak{sl}_2$ that $\ad h$ gives a good $\mathbb{Z}$-grading of $\g$ for $e$, called the {\em Dynkin grading}.
\end{example}

Assume that a good $\mathbb{Z}$-grading for $e$ is given. To simplify the definition of W-superalgebra, throughout this article, we assume in addition that the grading is {\em even}; that is, $\g(i)=0$ for all $i\notin 2\mathbb{Z}$. In the case of $\glMN$, an even good $\mathbb{Z}$-grading always exists; see Theorem \ref{Hoyt} below. Note that, in general, the Dynkin grading in the example above may not be even. 

\begin{remark}
Even good $\mathbb{Z}$-gradings do not always exist in other types; see \cite{Ho}. 
\end{remark}

Define the following subalgebras of $\g$ by
\begin{equation}\label{mpdef}
\mathfrak{p}:=\bigoplus_{j\geq 0}\g(j), \quad \mathfrak{m}:=\bigoplus_{j<0}\g(j).
\end{equation}
Let $\chi\in\g^*$ be defined by $\chi(y):=(y,e)$, for all $y\in\g$. Then the restriction of $\chi$ on $\mathfrak{m}$ gives a one dimensional $U(\mathfrak{m})$-module. Let $I_\chi$ denote the left ideal of $U(\g)$ generated by $\{a-\chi(a)|a\in\mathfrak{m}\}$. 

By the PBW theorem for $U(\g)$, we have $U(\g)=U(\mathfrak{p})\oplus I_\chi$. Let $\pr_\chi:U(\g)\rightarrow U(\mathfrak{p})$ be the natural projection and we identify $U(\g)/I_\chi \cong U(\mathfrak{p})$. Next we define the $\chi$-twisted action of $\mathfrak{m}$ on $U(\mathfrak{p})$ by  $a\cdot y := \pr_\chi([a,y])$ for all $a\in\mathfrak{m}$ and $y\in U(\mathfrak{p})$. 

The {\em finite W-superalgebra} (which we will usually omit the term ``finite" from now on) is defined to be the space of $\mathfrak{m}$-invariants of $U(\mathfrak{p})$ under the $\chi$-twisted action; that is,
\begin{align*}
\mathcal{W}_{e,h}:=U(\mathfrak{p})^\mathfrak{m}=&\{y\in U(\mathfrak{p})| \pr_\chi ([a,y])=0, \forall a\in\mathfrak{m}\}\\
=&\{y\in U(\mathfrak{p})| \big(a-\chi(a)\big)y\in I_\chi, \forall a\in\mathfrak{m}\}.
\end{align*}

\begin{example}
If we take the nilpotent element $e=0$, then $\chi=0$, $\g=\g(0)=\mathfrak{p}$, $\mathfrak{m}=0$, $\mathcal{W}_{e,h}=U(\mathfrak{p})=U(\g)$.
\end{example}

At this point, the definition of a finite $W$-superalgebra depends on the nilpotent element $e$ and a semisimple element $h$ which gives a good $\mathbb{Z}$-grading for $e$.


\subsection{Pyramids and finite $W$-superalgebras}
Suppose that we are given a tuple of positive integers $(q_1,\ldots,q_\ell)$. We associate a diagram $\pi$ consisting of $q_1$ boxes stacked in the first (leftmost) column, $q_2$ boxes stacked in the second column, $\ldots$ , $q_\ell$ boxes stacked in the $\ell$-th (rightmost) column. 

We call a diagram obtained in this way a $pyramid$ if each row of the diagram is a connected horizontal strip. For example, the left diagram is a pyramid but the right one is not:
$$
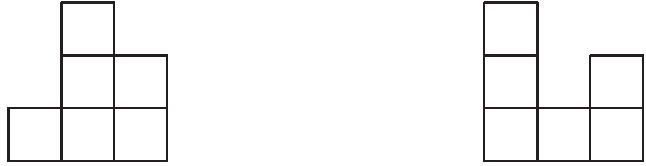
\begin{picture}(150,70)
\put(-60,10){\line(1,0){60}}
\put(-60,30){\line(1,0){60}}
\put(-40,50){\line(1,0){40}}
\put(-40,70){\line(1,0){20}}
\put(0,10){\line(0,1){40}}
\put(-60,10){\line(0,1){20}}
\put(-40,10){\line(0,1){40}}
\put(-20,10){\line(0,1){60}}
\put(-40,50){\line(0,1){20}}

\put(120,10){\line(1,0){60}}
\put(120,30){\line(1,0){60}}
\put(120,50){\line(1,0){20}}
\put(120,70){\line(1,0){20}}
\put(160,50){\line(1,0){20}}

\put(180,10){\line(0,1){40}}
\put(120,10){\line(0,1){60}}
\put(140,10){\line(0,1){60}}
\put(160,10){\line(0,1){40}}
\end{picture}
$$

Moreover, we assign the boxes of a given pyramid with $+$ and $-$ such that the boxes in a same row have the same sign. Such a pyramid will be called a {\em signed pyramid}. For example,
\[
\pi={\begin{picture}(90, 35)%
\put(15,-10){\line(1,0){60}}
\put(15,5){\line(1,0){60}}
\put(30,20){\line(1,0){45}}
\put(30,35){\line(1,0){30}}
\put(15,-10){\line(0,1){15}}
\put(30,-10){\line(0,1){45}}
\put(45,-10){\line(0,1){45}}
\put(60,-10){\line(0,1){45}}
\put(75,-10){\line(0,1){30}}
\put(18,-5){$-$}\put(33,-5){$-$}\put(48,-5){$-$}\put(63,-5){$-$}
\put(33,10){$-$}\put(48,10){$-$}\put(63,10){$-$}
\put(33,25){$+$}\put(48,25){$+$}
\end{picture}}\\[4mm]
\]

Let $M$ (respectively, $N$) be the number of boxes assigned with $+$ (respectively,~$-$) in a signed pyramid $\pi$. We now explain how to obtain an even nilpotent $e(\pi)$ and a semisimple $h(\pi)$ in $\glMN$ which gives a good $\mathbb{Z}$-grading for $e$ from a given signed pyramid $\pi$.

First of all, we enumerate the $``+"$-boxes by the numbers $\pa{1},\ldots, \pa{M}$ down columns from left to right, and enumerate the $``-"$-boxes by the numbers $1,\ldots,N$ in the same way. In fact, any numbering of $``+"$-boxes by the numbers $\pa{1},\ldots, \pa{M}$ and $``-"$-boxes by the numbers $1,\ldots,N$ would work , so we may choose the most intuitive and convenient way for our purpose. Moreover, we image that each box of $\pi$ is of size $2\times 2$ and our pyramid is built on the $x$-axis where the center of $\pi$ is exactly on the origin as shown in the example below.
\begin{equation}\label{exp}
\pi={\begin{picture}(90, 60)%
\put(15,-10){\line(1,0){80}}
\put(15,10){\line(1,0){80}}
\put(35,30){\line(1,0){60}}
\put(35,50){\line(1,0){40}}
\put(15,-10){\line(0,1){20}}
\put(35,-10){\line(0,1){60}}
\put(55,-10){\line(0,1){60}}
\put(75,-10){\line(0,1){60}}
\put(95,-10){\line(0,1){40}}
\put(23,-3){$1$}\put(43,-3){$3$}\put(63,-3){$5$}\put(83,-3){$7$}
\put(43,16){$2$}\put(63,16){$4$}\put(83,16){$6$}
\put(43,35){$\pa{1}$}\put(63,35){$\pa{2}$}
\put(0,-20){\line(1,0){110}}
\put(53,-23){$\bullet$}
\put(63,-35){$1$}\put(83,-35){$3$}
\put(35,-35){$-1$}\put(15,-35){$-3$}
\end{picture}}\\[10mm]
\end{equation}

Let $I=\{\pa{1}<\ldots<\pa{M}<1<\ldots<N\}$ be an ordered index set and let $\{e_i|i\in I\}$ be a basis of $\mathbb{C}^{M|N}$. We identify $\gl_{M|N}\cong\End (\mathbb{C}^{M|N})$ with the set of $(M+N)\times(M+N)$ matrices over $\mathbb{C}$ by this basis of $\mathbb{C}^{M|N}$ with respect to the following order 
\[e_i<e_j \text{ if } i<j \text{ in } I.\]

Define the element 
\begin{equation}\label{edef}
e(\pi):=\sum_{\substack{i,j\in I}}e_{i,j}\in \gl_{M|N},
\end{equation}
summing over all adjacent pairs $\young(ij)$ of boxes in $\pi.$ It is clear that such an element $e(\pi)$ is even nilpotent.

Let $\widetilde{\text{col}}(i)$ denote the $x$-coordinate of the box numbered with $i\in I$. Then we define the following diagonal matrix
\begin{equation}\label{hdef}
h(\pi):=-\text{diag}\big(\widetilde{\text{col}}(\pa{1}), \ldots, \widetilde{\text{col}}(\pa{M}), \widetilde{\text{col}}(1),\ldots, \widetilde{\text{col}}(N)\big)\in\gl_{M|N}.
\end{equation}
For example, the elements $e(\pi)$ and $h(\pi)$ in $\gl_{2|7}$ associated to the $\pi$ in (\ref{exp}) are
\begin{align*}
e(\pi)=e_{\pa{1}\,\pa{2}}+e_{24}+e_{46}+e_{13}+e_{35}+e_{57},\\
h(\pi)=\text{diag}(1,-1,3,1,1,-1,-1,-3,-3).
\end{align*}
One may check directly that $\ad h(\pi)$ gives a good $\mathbb{Z}$-grading of $\g$ for $e(\pi)$. 

Observe that if we horizontally shift the rows of $\pi$ to obtain another pyramid $\vec\pi$, then $e(\pi)$ and $e(\vec\pi)$ have the same Jordan type and hence they belong to precisely the same nilpotent orbit.
On the contrary, $h(\pi)\neq h(\vec\pi)$. The following theorem assures that every even good $\mathbb{Z}$-gradings for $e(\pi)$ is obtained by shifting the rows of $\pi$. In fact, it still holds if we drop the term ``even".

\begin{theorem}\cite[Theorem 7.2]{Ho}\label{Hoyt}
Let $\pi$ be a given signed pyramid, $e=e(\pi)$ and $h=h(\pi)$ be the elements in $\glMN$ defined by (\ref{edef}) and (\ref{hdef}), respectively. Then $\ad h$ defines an even good $\mathbb{Z}$\text{-}grading for $e$. Moreover, any even good $\mathbb{Z}$\text{-}grading for $e$ is defined by $\ad h(\vec\pi)$ where $\vec\pi$ is some signed pyramid obtained by shifting rows of $\pi$ horizontally.
\end{theorem}

Now we characterize those $e$ which is of the form $e(\pi)$ for some $\pi$. Let $e\in\gl_{M|N}$ be even nilpotent.
Consider 
\begin{equation}\label{edecomp}
e=e_M\oplus e_N\in\End\mathbb{C}^{M|N},
\end{equation}
where $e_M$ and $e_N$ are the restriction of $e$ to $\mathbb{C}^{M|0}$ and $\mathbb{C}^{0|N}$, respectively. Let $\mu=(\mu_1,\mu_2,\ldots)$ and $\nu=(\nu_1,\nu_2,\ldots)$ be the partitions representing the Jordan type of $e_M$ and $e_N$, respectively. We define a new partition $\la$ by collecting all parts of $\mu$ and $\nu$ together and reorder them by the usual decreasing order, except that if $\mu_i=\nu_j$ for some $i,j$, then we let $\mu_i$ appear first. 

For example, consider $\gl_{6|4}$, $\mu=(3,2,1)$ and $\nu=(3,1)$. Then the resulting $\la$ is given by 
$\lambda=(3^{+},3,2^{+},1^{+},1)$, where we use $i^{+}$ to indicate the number $i$ come from $\mu$.

In particular, the Young diagram $\la$ in French style (which means the longest row is in the bottom) is itself a signed pyramid if we assign the boxes of a row by $``+"$ when that row corresponds to a part of $\mu$ and by $``-"$ otherwise (so we do need to distinguish the numbers from $\mu$ and $\nu$). 

As a consequence, the good $\mathbb{Z}$-gradings of $\gl_{M|N}$ for {\em any} nilpotent $e$ are classified by Theorem~\ref{Hoyt}. Therefore, for a given signed pyramid $\pi$, we let $\W_\pi:=\W_{e(\pi),h(\pi)}$ denote the $W$-superalgebra associated to $e(\pi)$ and $h(\pi)$.
\begin{remark}
Under a certain assumption (see (\ref{crucialhypo}) below), we will prove eventually (Corollary \ref{ezg}) that up to isomorphism, $\W_\pi$ depends only on $e(\pi)$, not on the even good $\mathbb{Z}$-grading given by $h(\pi)$, which is not obvious from the definition.
\end{remark}

Let $\g=\gl_{M|N}$. Now we label the columns of $\pi$ from left to right by $1,\ldots, \ell$, and for any $i\in I$ we let col($i$) denote the column where $i$ appear. Define the {\em Kazhdan filtration} of $U(\g)$
\[
\cdots\subseteq F_dU(\g) \subseteq F_{d+1}U(\g)\subseteq \cdots
\] 
by declaring 
\begin{equation}\label{degdef}
\deg (e_{i,j}):= \text{col}(j)-\text{col}(i)+1
\end{equation}
for each $i,j \in I$ and $F_dU(\g)$ denotes the span of all supermonomials $e_{i_1,j_1}\cdots e_{i_s,j_s}$ for $s\geq 0$ and $\sum_{k=1}^s$ deg $(e_{i_k,j_k})\leq d$. Let $\gr U(\g)$ denote the associated graded superalgebra. The natural projection $\g\twoheadrightarrow\mathfrak{p}$ induces a grading on $\W_\pi$. 

On the other hand, let $\g^e$ denote the centralizer of $e$ in $\g$ and let $S(\g^e)$ denote the associated supersymmetric superalgebra. We define the Kazhdan filtration on $S(\g^e)$ by the same setting (\ref{degdef}). The following proposition was observed in \cite{Zh}, where the mild assumption there is satisfied when the good $\mathbb{Z}$-grading for $e$ is even.

\begin{proposition}\cite[Remark 3.9]{Zh}\label{dimprop}
$\gr \W_\pi\cong S(\g^e)$ as Kazhdan graded superalgebras.
\end{proposition}

\subsection{Pyramids and shifted super Yangians}

From now on and the remainder of this article, we assume that our signed pyramids satisfy the following property:
\begin{equation}\label{crucialhypo}
\text{The top row of } \pi \text{ is the only row assigned with +.} 
\end{equation}
In terms of the notations in ($\ref{edecomp}$), it means that $e(\pi)=e=e_M\oplus e_N$, where $e_M$ is principal nilpotent in $\gl_{M|0}$ and the sizes of the Jordan blocks of $e_N$ are all greater or equal to $M$. 

Let $\pi$ be a fixed signed pyramid satisfying (\ref{crucialhypo}). We explain how to associate $\pi$ with a truncated shifted super Yangian. Firstly, let $\ell$ be the length of the bottom row of $\pi$, and we label the columns of $\pi$ by $1,\ldots,\ell$ from left to right. Next we let $|q_i|$ denote the number of boxes in the $i$-th column of $\pi$ for $1\leq i\leq \ell$ and let $n+1$ be the maximal number in $\{ |q_i| \,|1\leq i\leq \ell\}$. Also, we label the rows of $\pi$ by $1,\ldots,n+1$ from top to bottom.

Define the shift matrix $\sigma:=(s_{i,j})_{1\leq i,j\leq n+1}$ by setting
\begin{align*}
s_{i,j}:= &\,\text{the number of bricks the }i\text{-th row is indented
from the }j\text{-th row}\\
 &\text{at the left (respectively, right) edge of}\\
 &\text{the diagram if }i \geq j\text{ (respectively, if }i\leq j).  
\end{align*}
It is not hard to check that such a definition gives a shift matrix satisfying (\ref{sijk}). With this shift matrix $\sigma$ and the integer $\ell$, we know precisely the generators and defining relations for $Y_{1|n}(\sigma)$ and hence $Y_{1|n}^\ell(\sigma)$ as its quotient.

This process can be reversed as follows. Let $\sigma=(s_{ij})_{1\leq i,j\leq n+1}$ be a shift matrix and let the level $\ell\geq s_{1,n+1}+s_{n+1,1}$ be given. We define the tuple $(p_1,\dots,p_{n+1})$ by setting $p_i := \ell-s_{i,n+1}-s_{n+1,i}$ as in (\ref{defpi}). Now draw a pyramid $\pi$ with $p_i$ bricks on the $i$th row indented $s_{n+1,i}$ columns from the left-hand edge and $s_{i,n+1}$ columns from the right-hand edge, for each $i=1,\dots,n+1$. Finally we assign the top row of $\pi$ by + and the remaining boxes by $-$. 

As a summary, given a signed pyramid $\pi$, we obtain $Y_{1|n}^\ell(\sigma)$, a shifted Yangian of level $\ell$ and vice versa. For example, 
$$
\ell=6,\,
\sigma = \left(\begin{array}{l|lll}
0&1&1&2\\
\hline
0&0&0&1\\
1&1&0&1\\
2&2&1&0
\end{array}\right)
\qquad\longleftrightarrow \qquad
\pi = 
{\begin{picture}(90, 35)%
\put(0,-25){\line(1,0){90}}
\put(0,-10){\line(1,0){90}}
\put(15,5){\line(1,0){60}}
\put(30,20){\line(1,0){45}}
\put(30,35){\line(1,0){30}}
\put(0,-25){\line(0,1){15}}
\put(15,-25){\line(0,1){30}}
\put(30,-25){\line(0,1){60}}
\put(45,-25){\line(0,1){60}}
\put(60,-25){\line(0,1){60}}
\put(75,-25){\line(0,1){45}}
\put(90,-25){\line(0,1){15}}
\put(3,-20){$-$}\put(18,-20){$-$}\put(33,-20){$-$}\put(48,-20){$-$}\put(63,-20){$-$}\put(78,-20){$-$}
\put(18,-5){$-$}\put(33,-5){$-$}\put(48,-5){$-$}\put(63,-5){$-$}
\put(33,10){$-$}\put(48,10){$-$}\put(63,10){$-$}
\put(33,25){$+$}\put(48,25){$+$}
\end{picture}}
$$

The following proposition is a well-known result about $\g^e$. As remarked in \cite{BBG}, the result is similar to the Lie algebra case $\gl_{M+N}$ since $e$ is even.
\begin{proposition}\label{counting2}
Let $\pi$ be a signed pyramid with row lengths $\lbrace p_i| 1\leq i\leq n+1\rbrace$, $\sigma=(s_{i,j})_{1\leq i,j\leq n+1}$ be the associated shift matrix of $\pi$ and let $e$ be the nilpotent element defined by (\ref{edef}). For all $1\leq i,j\leq n+1$ and $r>0$, define
\[
c_{i,j}^{(r)}:=
\sum_{\substack{h,k\in I \\ \row(h)=i, \row(k)=j\\ \col(k)-\col(h)=r-1}}e_{h,k}\in \g=\glMN.
\]
Then the set of vectors $\lbrace c_{i,j}^{(r)}|1\leq i,j\leq n+1, s_{i,j}<r\leq s_{i,j}+p_{min (i,j)}\rbrace$ forms a basis for $\g^e$.
\end{proposition}

\begin{corollary}\label{dimcoro}
Consider $Y_{1|n}^\ell(\sigma)$ with the canonical filtration and $S(\g^e)$ with the Kazhdan filtration. Let $F_dY_{1|n}^\ell(\sigma)$ and $F_dS(\g^e)$ denote the superspace with total degree $\leq d$ in the associated filtered superalgebras, respectively. Then for each $d\geq 0$, we have $\dim F_dY_{1|n}^\ell(\sigma) = \dim F_dS(\g^e)$.
\end{corollary}
\begin{proof}
Follows from Theorem \ref{PBWSYLpara}, Proposition \ref{counting2} and induction on $d$.
\end{proof}

\begin{remark}\label{gencase}
In the most general case, where the even nilpotent element $e$ could be {\it arbitrary}, one may still use a signed pyramid to denote $e$ and $h$ simultaneously; or equivalently a pair ($\sigma$, $\ell$) as what we did above. However, a presentation of the corresponding shifted super Yangian $Y_{m|n}(\sigma)$ (and its quotient) is unknown yet. Moreover, the descriptions of the maps $\Delta_R$ and $\Delta_L$ in Theorem \ref{baby1} and the computations in section 9 would be much more complicated if we drop the assumption (\ref{crucialhypo}).
\end{remark}

\section{Invariants}\label{Invariants}
Let $\pi$ be a given signed pyramid satisfying (\ref{crucialhypo}). We will define some elements in $U(\mathfrak{p})$. It turns out later that they are $\mathfrak{m}$-invariant, i.e., belong to $\W_\pi$.

Let $(\check{q}_1,\dots,\check{q}_{\ell})$ denote the {\em super column heights} of $\pi$, where each $\check{q}_i$ is defined to be the number of boxes signed with $``+"$ subtract the number of boxes signed with $``-"$ in the $i$-th column of $\pi$. We also define the {\em absolute height} by setting
$|q_i|$:= the number of total boxes (regardless the signs) in the $i$-th column. Let $h$ denote the {\em super height} of $\pi$, which is defined to be the number $\check{q}_j$ such that $|q_j|$ is maximal.

Define $\rho = (\rho_1,\dots,\rho_\ell)$ by setting that
\begin{equation}\label{rhodef}
\rho_r := h-\check{q}_{r} - \check{q}_{r+1} -\cdots-\check{q}_\ell
\end{equation}
for each $r=1,\dots,\ell$.

Recall the ordered index set $I:=\lbrace \pa{1}<\ldots<\pa{M}<1<\ldots<N \rbrace$, where $M$ and $N$ denote the number of boxes of $\pi$ assigned with ``$+$" and ``$-$", respectively. 
For all $i,j\in I$, define
\begin{equation}\label{etil}
\tilde e_{i,j} := (-1)^{\col(j)-\col(i)} 
(e_{i,j} + \delta_{i,j} (-1)^{\tpa(i)}\rho_{\col(i)}),
\end{equation}
where $\tpa(i):= 0$ if $i\in\{\pa{1},\ldots,\pa{M}\}$ and $\tpa(i):= 1$ otherwise. One should be careful that the number $\tpa(i)$ defined here is for $\gl_{M|N}$, while the number $\tp(i)$ defined in \textsection 2 is for $Y_{1|n}$.

One may check that \begin{multline}\label{etilrel}
[\tilde e_{i,j}, \tilde e_{h,k}]
=(\tilde{e}_{i,k} - \delta_{i,k} (-1)^{\tpa(i)} \rho_{\col(i)})\delta_{h,j}\\
- (-1)^{(\tpa(i)+\tpa(j))(\tpa(h)+\tpa(k))}
\delta_{i,k} (\tilde e_{h,j} - \delta_{h,j} (-1)^{\tpa(j)}\rho_{\col(j)}).
\end{multline}

Let us also spell out the effect of the homomorphism $U(\mathfrak{m}) \rightarrow \C$ induced by the character $\chi$. For any $i,j\in I$, we have
\begin{equation}\label{chidef}
\chi (\tilde e_{i,j}) = \left\{
\begin{array}{ll}
(-1)^{\tpa(i)+1}&\hbox{if $\row(i)=\row(j)$ and $\col(i) = \col(j)+1$;}\\
0&\hbox{otherwise.}
\end{array}\right.
\end{equation}

Now we give the most important definition of this article. For $1\leq i,j\leq n+1$ and signs
$\sigma_i\in \{\pm\}$,
we let
$T_{ i,j;\sigma_{1},\dots,\sigma_{n+1}}^{(0)} := \delta_{i,j} \sigma_i$
and for $r \geq 1$ define
\begin{equation}\label{tdef}
T_{i,j;\sigma_{1},\ldots,\sigma_{n+1}}^{(r)}
:=
\sum_{s = 1}^r
\sum_{\substack{i_1,\dots,i_s\\j_1,\dots,j_s}}
\sigma_{\row(j_1)} \cdots \sigma_{\row(j_{s-1})}
(-1)^{\tpa(i_1)+\cdots+\tpa(i_s)}
 \tilde e_{i_1,j_1} \cdots \tilde e_{i_s,j_s}
\end{equation}
where the second sum is taken over all $i_1,\dots,i_s,j_1,\dots,j_s\in I$
such that
\begin{itemize}
\item[(1)] $\deg(e_{i_1,j_1})+\cdots+\deg(e_{i_s,j_s}) = r$ (recall (\ref{degdef}));
\item[(2)] $\col(i_t) \leq \col(j_t)$ for each $t=1,\dots,s$;
\item[(3)] if $\sigma_{\row(j_t)} = +$ then
$\col(j_t) < \col(i_{t+1})$ for each
$t=1,\dots,s-1$;
\item[(4)]
if $\sigma_{\row(j_t)} = -$ then $\col(j_t) \geq \col(i_{t+1})$
for each
$t=1,\dots,s-1$;
\item[(5)] $\row(i_1)=i$, $\row(j_s) = j$;
\item[(6)]
$\row(j_t)=\row(i_{t+1})$ for each $t=1,\dots,s-1$.
\end{itemize}

Note that the assumptions (1) and (2) imply that $T_{i,j;\sigma_{1},\dots,\sigma_{n+1}}^{(r)}$ belongs to $\mathrm{F}_r U(\mathfrak p)$. For $0\leq x\leq n+1$, let $T_{i,j;x}^{(r)}$ denote $T_{i,j;\sigma_{1},\dots,\sigma_{n+1}}^{(r)}$ in the special case that $\sigma_{i}=-$ for all $i\leq x$ and $\sigma_{j}=+$ for all $j>x$.
Define the following series for all $1\leq i,j\leq n+1$:
\begin{equation}\label{tseries}
T_{i,j;x}(u) := \sum_{r \geq 0} T_{i,j;x}^{(r)} u^{-r}
\in U(\mathfrak p) [[u^{-1}]].
\end{equation}

\begin{lemma}\label{ttodef}
Let $i,j,x,y$ be non-negative integers.
\begin{itemize}
\item[(i)] If $x < i \leq y < j \leq n+1$ then
$$
T_{i,j;x}(u) = \sum_{k=x+1}^y T_{i,k;x}(u) \, T_{k,j;y}(u).
$$
\item[(ii)] If $x < j \leq y<i\leq n+1$ then
$$
T_{i,j;x}(u) = \sum_{k=x+1}^y  T_{i,k;y}(u) \, T_{k,j;x}(u).
$$
\item[(iii)] If $x< y < i \leq n+1$ and $y < j \leq n+1$, then
$$
T_{i,j;x}(u) = T_{i,j;y}(u)
+ \sum_{k,l=x+1}^y  T_{i,k;y}(u) \, T_{k,l;x}(u) \, T_{l,j;y}(u).
$$
\item[(iv)]
If $x < i \leq y\leq n+1$ and $x < j \leq y$, then
$$
\sum_{k=x+1}^y  T_{i,k;x}(u) \, T_{k,j;y}(u) = -\delta_{i,j}.
$$
\end{itemize}
\end{lemma}

\begin{proof}
All of the proofs are exactly the same as in \cite[Lemma 9.2]{BK2} and hence they are omitted. 
\end{proof}

Define $T(u) := \big( T_{i,j;0}(u)\otimes (-1)^{\tp(j) (\tp(i)+1)}E_{i,j} \big)_{1\leq i,j\leq n+1}$, an invertible $(n+1)\times (n+1)$ matrix with entries in $U(\mathfrak{p})[[u^{-1}]]$ where $E_{i,j}$ denotes the elementary matrices. Also let $\mu = (\mu_1\,|\,\mu_{2},\ldots,\mu_{m+1})$ be a fixed shape such that $\mu_1=1$. Consider the Gauss factorization $T(u) = F(u)D(u)E(u)$ where $D(u)$ is a block diagonal matrix, $E(u)$ is a block upper unitriangular matrix, and $F(u)$ is a block lower unitriangular matrix, all block matrices being of shape $\mu$. The diagonal blocks of $D(u)$ define matrices $D_1(u),\ldots,D_{m+1}(u)$, the upper diagonal blocks of $E(u)$ define matrices $E_1(u),\ldots,E_{m}(u)$, and the lower diagonal matrices of $F(u)$ define matrices $F_1(u),\dots,F_{m}(u)$. Also let $D^{\prime}_a(u) := D_a(u)^{-1}$.

Thus $D_a(u) = (D_{a;i,j}(u))_{1 \leq i,j \leq \mu_a}$ and $D^{\prime}_a(u) = (D^{\prime}_{a;i,j}(u))_{1 \leq i,j \leq \mu_a}$ are $\mu_a \times \mu_a$ matrices, $E_a(u) = (E_{a;i,j}(u))_{1 \leq i \leq \mu_a, 1 \leq j \leq \mu_{a+1}}$ are $\mu_a \times \mu_{a+1}$ matrices, and $F_a(u) = (F_{a;i,j}(u))_{1 \leq i \leq \mu_{a+1}, 1 \leq j \leq \mu_{a}}$ are $\mu_{a+1} \times \mu_{a}$ matrices, respectively.
Write
\begin{align*}
D_{a;i,j}(u) &= \sum_{r \geq 0} D_{a;i,j}^{(r)} u^{-r},\quad
& D^{\prime}_{a;i,j}(u) &= \sum_{r \geq 0} 
D_{a;i,j}^{\prime(r)} u^{-r},\\
E_{a;i,j}(u) &= \sum_{r > 0} E_{a;i,j}^{(r)} u^{-r},\quad
&F_{a;i,j}(u) &= \sum_{r > 0} F_{a;i,j}^{(r)} u^{-r},
\end{align*}
and then the elements $D_{a;i,j}^{(r)}, E_{a;i,j}^{(r)}$ and $F_{a;i,j}^{(r)}$ of $U(\mathfrak{p})$ are defined, all depending on the fixed choice of $\mu$. All of them are parallel to the definition of the elements of $Y_{1|n}$ with the same names in \textsection 3.

\begin{theorem}\label{ttodefthm}
With $\mu = (\mu_1\,|\,\mu_{2},\ldots,\mu_{m+1})$ fixed as above and all admissible
$a,i,j$, we have that
\begin{align*}
D_{a;i,j}(u) &= T_{\mu_1+\cdots+\mu_{a-1}+i,\mu_1+\cdots+\mu_{a-1}+j;
\mu_1+\cdots+\mu_{a-1}}(u),\\
D^{\prime}_{a;i,j}(u) &=
-T_{\mu_1+\cdots+\mu_{a-1}+i,\mu_1+\cdots+\mu_{a-1}+j;
\mu_1+\cdots+\mu_a}(u),\\
E_{a;i,j}(u) &= T_{\mu_1+\cdots+\mu_{a-1}+i,\mu_1+\cdots+\mu_{a}+j;
\mu_1+\cdots+\mu_{a}}(u),\\
F_{a;i,j}(u) &= T_{\mu_1+\cdots+\mu_{a}+i,\mu_1+\cdots+\mu_{a-1}+j;
\mu_1+\cdots+\mu_{a}}(u).
\end{align*}
\end{theorem}

\begin{proof}
Note that it is enough to show the formulae for $D, E$ and $F$, since the one for $D^\prime$ follows from the one for $D$ and Lemma \ref{ttodef}(iv). We prove this by induction on the length of $\mu$. The initial case is $\mu=(1\,|\,n)$, a composition of length 2. By Gauss decomposition, we have
$$
T(u)
=
\left(\begin{array}{ll}I_1&0\\ F_1&I_n\end{array}\right)
\left(\begin{array}{ll}D_1&0\\0&D_2\end{array}\right)
\left(\begin{array}{ll}I_1&E_1\\0&I_n\end{array}\right)
= 
\left(
\begin{array}{cc}
D_1&D_1E_1\\
F_1D_1&D_2+F_1D_1E_1
\end{array}\right).
$$
Comparing the corresponding blocks and using Lemma \ref{ttodef}, we have proved the initial step of the induction.

Now let $\mu=(\mu_1\,|\,\mu_{2},\ldots,\mu_m,\mu_{m+1})$ with $m\geq2$ be a composition of length $m+1$ be given.
Define a new composition $\nu=(\nu_1\,|\,\ldots,\nu_m)$ of length $m$ by setting $\nu_i=\mu_i$ for all $1\leq i\leq m-1$ and $\nu_m = \mu_m+\mu_{m+1}$. By the induction hypothesis, we have
\begin{align*}
{^\nu}D_{a}(u) &= \left( T_{\nu_1+\cdots+\nu_{a-1}+i,\nu_1+\cdots+\nu_{a-1}+j;
\nu_1+\cdots+\nu_{a-1}}(u)\right)_{1 \leq i,j \leq \nu_a},\, \forall \,1\leq a \leq m,\\
{^\nu}E_{a}(u) &= \left( T_{\nu_1+\cdots+\nu_{a-1}+i,\nu_1+\cdots+\nu_{a}+j;
\nu_1+\cdots+\nu_{a}}(u)\right)_{1 \leq i \leq \nu_a, 1 \leq j \leq \nu_{a+1}},\,\forall \,1\leq a\leq m-1,\\
{^\nu}F_{a}(u) &= \left( T_{\nu_1+\cdots+\nu_{a}+i,\nu_1+\cdots+\nu_{a-1}+j;
\nu_1+\cdots+\nu_{a}}(u)\right)_{1 \leq i \leq \nu_{a+1}, 1 \leq j \leq \nu_{a}},\,\forall \,1\leq a\leq m-1,
\end{align*}
where we add a superscript $\nu$ to emphasize that these elements are defined with respect to $\nu$. Note that 
${^\nu}D_{a}(u) = {^\mu}D_{a}(u)$ for all $1\leq a\leq m-1$ and ${^\nu}E_{a}(u) = {^\mu}E_{a}(u)$, ${^\nu}F_{a}(u) = {^\mu}F_{a}(u)$ for all $1\leq a \leq m-2$. Hence it is enough to show the explicit descriptions of the matrices ${^\mu}D_m(u)$, ${^\mu}D_{m+1}(u)$, ${^\mu}E_{m}(u)$ and ${^\mu}F_{m}(u)$ are as described in our theorem.

Define matrices $A,B,C$ and $D$ by
\begin{align*}
A &= \left( T_{\nu_1+\cdots+\nu_{m-1}+i,\nu_1+\cdots+\nu_{m-1}+j;\nu_1+\cdots+\nu_{m-1}}(u)\right)_{1 \leq i,j \leq \mu_m},\\
B &= \left( T_{\nu_1+\cdots+\nu_{m-1}+i,\nu_1+\cdots+\nu_{m-1}+\mu_m+j;\nu_1+\cdots+\nu_{m-1}+\mu_m}(u)\right)_{1 \leq i \leq \mu_m, 1 \leq j \leq \mu_{m+1}},\\
C &= \left( T_{\nu_1+\cdots+\nu_{m-1}+\mu_m+i,\nu_1+\cdots+\nu_{m-1}+j;\nu_1+\cdots+\nu_{m-1}+\mu_m}(u)\right)_{1 \leq i \leq \mu_{m+1}, 1 \leq j \leq \mu_m},\\
D &= \left( T_{\nu_1+\cdots+\nu_{m-1}+\mu_m+i,\nu_1+\cdots+\nu_{m-1}+\mu_m+j;\nu_1+\cdots+\nu_{m-1}+\mu_m} (u)\right)_{1 \leq i,j \leq \mu_{m+1}}.
\end{align*}
By Lemma \ref{ttodef} with $x=\mu_1+\ldots+\mu_{m-1}$ and $y=\mu_1+\ldots+\mu_{m}$, we have
$$
^{\nu}D_m(u)
=
\left(\begin{array}{ll}I_{\mu_m}&0\\ C&I_{\mu_{m+1}}\end{array}\right)
\left(\begin{array}{ll} A&0\\0& D\end{array}\right)
\left(\begin{array}{ll}I_{\mu_m}& B\\0&I_{\mu_{m+1}}\end{array}\right)
= 
\left(
\begin{array}{cc}
 A & AB\\
 CA & D+CAB
\end{array}\right).
$$
Now the explicit descriptions of the matrices $^{\mu}D_m(u)$, $^\mu D_{m+1}(u)$, $^\mu E_{m}(u)$ and $^\mu F_{m}(u)$ follows from Lemma \ref{split}, which completes the induction argument.
\end{proof}

In the extreme case that $\mu = (1^{n+1})$, we write simply $D_i^{(r)}, D_i^{\prime(r)},E_i^{(r)}$ and $F_i^{(r)}$ for the elements $D_{i;1,1}^{(r)}$, $D_{i;1,1}^{\prime(r)}$, $E_{i;1,1}^{(r)}$ and $F_{i;1,1}^{(r)}$ of $U(\mathfrak{p})$, respectively.

\begin{corollary}\label{transcor}
$D_i^{(r)} = T_{i,i;i-1}^{(r)}$, $E_i^{(r)} = T_{i,i+1;i}^{(r)}$, $F_i^{(r)} = T_{i+1,i;i}^{(r)}$ and $D_i^{\prime(r)} = -T_{i,i;i}^{(r)}$,.
\end{corollary}

\section{Main theorem}\label{mainsec}
 Let $\pi$ be a signed pyramid of base $\ell$ satisfying (\ref{crucialhypo}) and let $\sigma = (s_{i,j})_{1 \leq i,j \leq n+1}$ be the shift matrix associated to $\pi$ as explained in \textsection 7. Let $Y_{1|n}^\ell(\sigma)$ denote the truncated shifted Yangian associated to $\pi$ equipped with the canonical filtration and let $\W_\pi$ denote the finite $W$-superalgebra associated to $\pi$ equipped with the Kazhdan filtration .

Suppose also that $\mu = (\mu_1\,|\,\mu_2, \dots, \mu_{m+1})$ with $\mu_1=1$ is an admissible shape for $\sigma$, and recall the notation $s_{a,b}^\mu$ and $p_a^\mu$ from (\ref{sabmu}) and (\ref{pamu}). We have the elements $D_{a;i,j}^{(r)}$, $D_{a;i,j}^{\prime(r)}$, $E_{a;i,j}^{(r)}$ and $F_{a;i,j}^{(r)}$ of $U(\mathfrak{p})$ defined by Theorem~\ref{ttodefthm} relative to this fixed shape $\mu$. We also have the parabolic generators $D_{a;i,j}^{(r)}$, $D_{a;i,j}^{\prime(r)}$, $E_{a;i,j}^{(r)}$ and $F_{a;i,j}^{(r)}$ of $Y_{1|n}^\ell(\sigma)$ as defined in \textsection\ref{parabo}. The main result of the article is as follows.

\begin{theorem}\label{main}
Let $\pi$ be a signed pyramid satisfying (\ref{crucialhypo}). There is a unique isomorphism $Y_{1|n}^\ell(\sigma) \stackrel{\sim}{\rightarrow} \W_\pi$ of filtered superalgebras such that for any admissible shape $\mu = (\mu_1\,|\,\mu_2,\dots,\mu_{m+1})$ with $\mu_1=1$, the generators
\begin{align*}
&\{D_{a;i,j}^{(r)}\}_{1 \leq a \leq m+1,1 \leq i,j \leq \mu_a, r > 0},\\
&\{E_{a;i,j}^{(r)}\}_{1 \leq a < m+1, 1 \leq i \leq \mu_a, 1 \leq j \leq \mu_{a+1}, r > s_{a,b}^\mu },\\
&\{F_{a;i,j}^{(r)}\}_{1 \leq a < m+1, 1 \leq i \leq \mu_{a+1}, 1 \leq j \leq \mu_{a}, r > s_{b,a}^\mu}
\end{align*}
of $Y_{1|n}^\ell(\sigma)$ are mapped to the elements of $U(\mathfrak{p})$ with the same names. In particular, these elements of $U(\mathfrak{p})$ are $\mathfrak{m}$-invariants and they generate $\W_\pi$.
\end{theorem}

The rest of this article is devoted to prove Theorem \ref{main}. We shall prove it by induction on the number $\ell-t$, where $\ell$ is the length of the bottom row and $t$ is the length of the top row of $\pi$. 

The initial step $\ell=t$ was established in \cite{BR}; see also \cite{Pe2} for an approach similar to our setting here. In this case, the pyramid is of rectangular shape so the associated shift matrix is the zero matrix. Hence the shifted super Yangian is the whole super Yangian $Y_{1|n}$ itself, and its quotient is exactly the truncated super Yangian $Y_{1|n}^\ell$.

By \cite[Theorem 4.3]{Pe2}, the map $\gamma$ sending $t_{i,j}^{(r)}\in Y_{1|n}^\ell$ to $T_{i,j;0}^{(r)}\in\W_\pi$ defined by (\ref{tdef}) is an isomorphism of filtered superalgebras. By Lemma~\ref{ttodef} and Theorem~\ref{ttodefthm}, the images of the parabolic generators $$\{ D_{1}^{(r)}, D_{2;i,j}^{(r)}, E_{1;1,j}^{(s)}, F_{1;i,1}^{(s)} \,|\, 1\leq i,j\leq n, \,r\geq 0, \,s\geq1\}$$ in $Y_{1|n}^\ell$ under $\gamma$ are exactly the elements in $\W_\pi$ with the same name, which proves the initial step of the induction argument.

Now we assume that our signed pyramid $\pi$ is not a rectangle so $\ell-t>0$ and $\ell\geq 2$ ($\ell$=1 must be a rectangle). The first reduction is that it suffices to prove the case where $\mu$ is a minimal admissible shape for the shift matrix $\sigma$ associated to $\pi$, by induction on the length of the shape and Lemma \ref{split}. Therefore, we assume from now on that $\mu$ is a minimal admissible shape for $\sigma$ and we denote by $\beta$ the absolute column height of the shortest column of $\pi$. Since $\pi$ is a pyramid, we know that either $\beta=|q_1|$ or $\beta=|q_l|$, and we discuss them case-by-case.

\begin{itemize}
\item Case R: $|q_1|\geq |q_\ell|=\beta$.
\item Case L: $|q_1|=\beta< |q_\ell|$.
\end{itemize}

We will explain the proof of case R in detail and sketch the proof of case L, which can be obtained by a very similar argument. Our approach is similar to \cite{BK2}. Recall that we numbered the boxes of $\pi$ using the index set $I:=\lbrace \pa{1}<\ldots<\pa{M}<1<\ldots<N\rbrace$ in the standard way: down columns from left to right, where $\pa{i}$ (respectively, $i$) stands for the boxes assigned with $+$ (respectively, $-$).

Let $\dot\pi$ be the pyramid obtained by removing the rightmost column of $\pi$, i.e. removing the boxes numbered with $N-\beta+1,N-\beta+2,\dots,N$ of $\pi$. Let $\dot\sigma = (\dot s_{i,j})_{1 \leq i,j \leq n+1}$ be the shift matrix defined by ($\ref{babyr1}$) where its associated pyramid is $\dot\pi$. Define $\dot{\mathfrak{p}}, \dot{\mathfrak{m}}$ and $\dot e$ in $\dot{\mathfrak{g}} = \mathfrak{gl}_{M|N-\beta}$ according to (\ref{mpdef}) and (\ref{edef}) and let $\dot\chi:\dot{\mathfrak{m}}\rightarrow \C$ be the character $x \mapsto (x,\dot e)$.

Let $\dot D_{a;i,j}^{(r)}, \dot{D}_{a;i,j}^{\prime(r)}$, $\dot E_{a;i,j}^{(r)}$ and $\dot F_{a;i,j}^{(r)}$ denote the elements of $U(\dot{\mathfrak{p}})$ as defined in \textsection \ref{Invariants} associated to the same shape $\mu$. Note that $\mu$ is an admissible shape for both $\sigma$ and $\dot \sigma$. By the induction hypothesis, Theorem~\ref{main} holds for $\dot\pi$, so the following elements of $U(\dot{\mathfrak{p}})$ are invariant under the twisted action of $\dot{\mathfrak{m}}$,  i.e. they belong to the finite $W$-superalgebra $\W_{\dot\pi} = U(\dot{\mathfrak{p}})^{\dot{\mathfrak{m}}}$:
\begin{align*}
&\{\dot D_{a;i,j}^{(r)} \}\text{ and } \{\dot{D}_{a;i,j}^{\prime(r)}\} \text{ for } 1 \leq a \leq m+1, 1 \leq i,j \leq \mu_a \text{ and } r > 0;\\
&\{\dot E_{a;i,j}^{(r)}\} \text{ for } 1 \leq a \leq m, 1 \leq i \leq \mu_a, 1 \leq j \leq \mu_{a+1} \text{ and } r > s_{a,a+1}^\mu - \delta_{a,m};\\
&\{\dot F_{a;i,j}^{(r)} \}\text{ for } 1 \leq a \leq m, 1 \leq i \leq \mu_{a+1}, 1 \leq j \leq \mu_{a} \text{ and } r > s_{a+1,a}^\mu.
\end{align*}

We embed $U(\dot\g)$ into $U(\g)$ such that the generators $\tilde e_{ij}$ of $U(\dot \g)$ defined from $\dot\pi$ are mapped to the generators $\tilde e_{ij}$ of $U(\g)$ defined from $\pi$, for all $i,j$ in the index set $\dot I:=\lbrace \pa{1},\ldots,\pa{M},1,\ldots,N-\beta\rbrace$. It also embeds $U(\dot{\mathfrak{p}})$ into $U(\mathfrak{p})$ and $\dot{\mathfrak{m}}$ into $\mathfrak{m}$. Moreover, the character $\dot\chi$ of $\dot{\mathfrak{m}}$ is the restriction of the character $\chi$ of $\mathfrak{m}$ and hence the twisted action of $\dot{\mathfrak{m}}$ on $U(\dot{\mathfrak{p}})$ is the restriction of the twist action of $\mathfrak{m}$ on $U(\mathfrak{p})$.

The next crucial lemma gives the relations between the elements $D_{a;i,j}^{(r)}$, $E_{a;i,j}^{(r)}$, $F_{a;i,j}^{(r)}$ of $U(\mathfrak{p})$ and the elements $\dot D_{a;i,j}^{(r)}$, $\dot E_{a;i,j}^{(r)}$, $\dot F_{a;i,j}^{(r)}$ of $U(\dot{\mathfrak{p}})$.

\begin{lemma}\label{sbabyr1}
The following equations hold for $r > 0$, all admissible $a,i,j$ and any fixed $1 \leq h \leq \beta$:
\begin{align}\label{sbr11}\notag
D_{a;i,j}^{(r)} & = \dot D_{a;i,j}^{(r)}\\
 &+\delta_{a,m+1}
\left(-\sum_{k=1}^\beta
\dot D_{a;i,k}^{(r-1)} \tilde e_{N-\beta+k,N-\beta+j}
+
[\dot D_{a;i,h}^{(r-1)}, \tilde e_{N-2\beta+h,N-\beta+j}]\right),\\
E_{a;i,j}^{(r)} & = \dot E_{a;i,j}^{(r)} + \delta_{a,m}\left(
-\sum_{k=1}^\beta \dot E_{a;i,k}^{(r-1)} \tilde e_{N-\beta+k,N-\beta+j}
+ [\dot E_{a;i,h}^{(r-1)}, \tilde e_{N-2\beta+h,N-\beta+j}]\right),\label{sbr12}\\
F_{a;i,j}^{(r)} & = \dot F_{a;i,j}^{(r)},
\end{align}
where for (\ref{sbr12}) we are assuming that $r > s_{m,m+1}^\mu$ if $a=m$.
\end{lemma}

\begin{proof}
It follows from Theorem~\ref{ttodefthm} and the explicit form of the elements $T_{i,j;x}^{(r)}$ from (\ref{tdef}).
\end{proof}

In the next several lemmas we will use the above inductive descriptions and the induction hypothesis to show that the elements $\Daij$, $\Eaij$ and $\Faij$ of $U(\mathfrak{p})$ are $\mathfrak{m}$-invariants for the appropriate $r$'s.

\begin{lemma}\label{demint1}
The following elements of $U(\mathfrak{p})$ are $\mathfrak{m}$-invariant:
\begin{itemize}
\item[(i)] $\Daij$ and $\Dpaij$
for  $1 \leq a \leq m$, $1 \leq i,j \leq \mu_a$ and $r > 0$;
\item[(ii)] $\Eaij$ for 
$1 \leq a < m$, $1 \leq i \leq \mu_a$, $1 \leq j \leq \mu_{a+1}$ and $r>s_{a,a+1}^\mu$;
\item[(iii)] $\Faij$ for $1 \leq a \leq m$, $1 \leq i \leq \mu_{a+1}$, $1 \leq j \leq \mu_{a}$ 
and $r > s_{a+1,a}^\mu$.
\end{itemize}
\end{lemma}
\begin{proof}
By Lemma \ref{sbabyr1} and the definitions of $\Dpaij$ and $\dot{D}_{a;i,j}^{\prime(r)}$, all these elements of $U(\mathfrak{p})$ coincide with the corresponding elements of $U(\dot{ \mathfrak{p}})$. By the induction hypothesis, they are $\dot{\mathfrak{m}}$-invariant. Hence it is enough to show that they are invariant under the twisted action of all $\widetilde e_{f,g}$ in $\mathfrak{\dot{m}^c}$, which means the vector space complement of $\dot{\mathfrak{m}}$ in $\mathfrak{m}$. One should note that $\widetilde e_{f,g}\in \mathfrak{\dot{m}^c}$ if and only if $g\leq N-\beta <f\leq N$.

By Theorem \ref{ttodefthm} and the explicit form of (\ref{tdef}), all these elements under our consideration are linear combinations of supermonomials of the form $\tilde e_{i_1,j_1} \cdots \tilde e_{i_r,j_r}$ in $U(\dot{\mathfrak{p}})$ with $i_s\in \dot I$ and $\pa{1} \leq j_s \leq N-2\beta$ for all $s=1,\dots,r$.

By the fact that $\chi(\tilde{e}_{f,g}) = 0$ for all $g \leq N-2\beta$ and $N-\beta < f \leq N$, one may prove that all such supermonomials are invariant under the twisted action of all $\widetilde e_{f,g}\in \mathfrak{m^c}$ and our assertion follows.
\end{proof}

\begin{lemma}\label{demdot}
The following elements of $U(\mathfrak{p})$ are $\dot{\mathfrak{m}}$-invariant under the twisted action:
 \begin{enumerate}
  \item  $D_{m+1;i,j}^{(r)}$ for $1\leq i,j\leq \mu_{m+1}$ and $r>0$.
  \item  $E_{m;i,j}^{(r)}$ for $1\leq i\leq \mu_{m}$, $1\leq j\leq \mu_{m+1}$ and $r>s_{m,m+1}^\mu$.
 \end{enumerate}
\end{lemma}

\begin{proof}
(1) Let $x\in\dot{\mathfrak{m}}$. By (\ref{sbr11}), we have 
\[
D_{m+1;i,j}^{(r)}  = \dot D_{m+1;i,j}^{(r)} 
-\sum_{k=1}^\beta
\dot D_{m+1;i,k}^{(r-1)} \tilde e_{N-\beta+k,N-\beta+j}+
[\dot D_{m+1;i,h}^{(r-1)}, \tilde e_{N-2\beta+h,N-\beta+j}].
\]
Note that $[x,\widetilde e_{N-\beta+k,N-\beta+j}]=0=[x,\widetilde e_{N-2\beta+h,N-\beta+j}]$. Using this and the induction hypothesis, one can show that $\pr_\chi([x,D_{m+1;i,j}^{(r)}])=0$. (2) can be derived in a similar way.
\end{proof}

\begin{lemma}\label{de12dotminv}
 \begin{enumerate}
  \item  $D_{m+1;i,j}^{(1)}$ is $\mathfrak{\dot{m}^c}$-invariant for $1\leq i,j\leq \mu_{m+1}$.
  \item  Suppose $s_{m,m+1}^{\mu}=1$. 
  Then $D_{m+1;i,j}^{(2)}$ is $\mathfrak{\dot{m}^c}$-invariant 
  for $1\leq i,j\leq \mu_{m+1}$. 
  \item  Suppose $s_{m,m+1}^{\mu}=1$. 
  Then $E_{m;i,j}^{(2)}$ is $\mathfrak{\dot{m}^c}$-invariant 
  for $1\leq i\leq \mu_{m}$ and $1\leq j\leq \mu_{m+1}$.
 \end{enumerate}
\end{lemma}
\begin{proof}
(1) By (\ref{sbr11}), Theorem \ref{ttodefthm} and (\ref{tdef}), we have
\begin{equation*}
D_{m+1;i,j}^{(1)}=\dot{D}_{m+1;i,j}^{(1)}-\tilde{e}_{N-\beta+i,N-\beta+j}
=\sum_{\substack{\text{row}(p_k)=\mu_1+\ldots+\mu_m+i\\\text{row}(q_k)=\mu_1+\ldots+\mu_m+j\\\text{col}(p_k)=\text{col}(q_k)=k\\1\leq k\leq \ell-1}}
(-\tilde{e}_{p_k,q_k})-\tilde{e}_{N-\beta+i,N-\beta+j}.
\end{equation*}
Let $\tilde{e}_{f,g}\in\dot{\mathfrak{m^c}}$, hence $g\leq N-\beta<f\leq N$. 

If row($g$)$\neq\mu_1+\ldots+\mu_m+i$, then
$[\tilde{e}_{f,g},-\sum\tilde{e}_{p_k,q_k}]=0$ since $g\neq p_k$ and $f\neq q_k$ for any $p_k$, $q_k$ appearing in the sum. Also, $[\tilde{e}_{f,g},\tilde{e}_{N+\beta-i,N+\beta-j}]=-\del_{f,N-\beta+j}\tilde{e}_{N-\beta+i,g}$, which belongs to the kernel of $\chi$ by (\ref{chidef}). In this case, $\pr_\chi([\tilde{e}_{f,g}, D_{m+1;i,j}^{(1)}])=0$.

Assume from now on that row($g$)=$\mu_1+\ldots+\mu_m+i$. Then $g$ equals exactly one $p_k$ appearing in the sum and hence $[\tilde{e}_{f,g},-\sum \tilde{e}_{p_k,q_k}]=-\tilde{e}_{f,q_k}$ for a certain $1\leq k\leq \ell-1$, and it belongs to the kernel of $\chi$ unless col($q_k$)=$\ell-1$ by (\ref{chidef}). Also, $[\tilde{e}_{f,g},\tilde{e}_{N+\beta-i,N+\beta-j}]=-\del_{f,N-\beta+j}\tilde{e}_{N-\beta+i,g}$, which belongs to the kernel of $\chi$ except that col($g$)=$\ell-1$. Thus, $[\tilde{e}_{f,g}, D_{m+1;i,j}^{(1)}]$ belongs to the kernel of $\chi$ unless col($g$)=$\ell-1$ and row($g$)=$\mu_1+\ldots+\mu_m+i$, and this exception happens only when $g=N-2\beta+i$.

Therefore, we directly compute that 
\[[\tilde{e}_{f,N-2\beta+i},D_{m+1;i,j}^{(1)}]=-\tilde{e}_{f,N-2\beta+j}+\del_{f,N-\beta+j}\tilde{e}_{N-\beta+i,N-2\beta+i},\]
which belongs to the kernel of $\chi$ by (\ref{chidef}). As a result, $D_{m+1;i,j}^{(1)}$ is $\dot{\mathfrak{m^c}}$-invariant.

(2) and (3) can be derived similarly, although the computations are more involved.
\end{proof}

\begin{lemma}\label{derdotminv}
Suppose that $s_{m,m+1}^{\mu}=1$. Then the following identities hold in $U(\mathfrak{p})$ for $r>1$:
 \begin{enumerate}
   \item  \[E_{m;i,j}^{(r+1)}=(-1)^{\tp(m)}[D_{m;i,g}^{(2)},E_{m;g,j}^{(r)}]-\sum_{f=1}^{\mu_m}D_{m;i,f}^{(1)}E_{m;f,j}^{(r)}\,,\]
   \item  \[D_{m+1;i,j}^{(r+1)}=(-1)^{\tp(m)}[F_{m;i,g}^{(2)},E_{m;g,j}^{(r)}]-\sum_{t=1}^{r+1}D_{m+1;i,j}^{(r+1-t)}D_{m;g,g}^{\prime (r)}\,.\]
 \end{enumerate}
\end{lemma}
\begin{proof}
There are two possibilities here: either $m=1$ and $\mu=(1\,|\,\beta)$, or $m>1$ and $\mu=(1\,|\,\mu_2,\ldots ,\mu_m,\beta)$.
Assume that $m>1$. We prove (2) in detail, where (1) can be proved using exactly the same method.

By the induction hypothesis and defining relation (\ref{ef=dd}), we have
\begin{equation}\label{1221}
[\dot{F}_{m;i,g}^{(2)}, \dot{E}_{m;g,j}^{(r)}]=-(\sum_{t=0}^{r+1}\dot{D}_{m+1;i,j}^{(r+1-t)}\dot{D}_{m;g,g}^{\prime(t)})
=-\dot{D}_{m+1;i,j}^{(r+1)}-\sum_{t=1}^{r+1}\dot{D}_{m+1;i,j}^{(r+1-t)}\dot{D}_{m;g,g}^{\prime(t)}.
\end{equation} 
By Lemma \ref{sbabyr1}, we have
\begin{equation}\label{1222}
E_{m;g,j}^{(r)}=\dot{E}_{m;g,j}^{(r)}-\sum_{k=1}^{\beta}\dot{E}_{m;g,k}^{(r-1)}\tilde{e}_{N-\beta+k,N-\beta+j}+[\dot{E}_{m;g,h}^{(r-1)},\tilde{e}_{N-2\beta+h,N-\beta+j}].
\end{equation}
Now bracket (\ref{1222}) with $F_{m;i,g}^{(2)}=\dot{F}_{m;i,g}^{(2)}$. 
Note that no supermonomial in the expansion of $\dot{F}_{m;i,g}^{(2)}$ 
contains any matrix unit of the forms $\tilde{e}_{?,N-\beta+h}$, $\tilde{e}_{N-\beta+h,?}$ or $\tilde{e}_{N-2\beta+h,?}$. 
As a result, $[\dot{F_{m;i,g}^{(2)}},\tilde{e}_{N-\beta+j,N-\beta+k}]=[\dot{F_{m;i,g}^{(2)}},\tilde{e}_{N-2\beta+h,N-\beta+j}]=0$.

Using (\ref{1221}) and (\ref{sbr11}) twice, we obtain
\begin{align*}
[F_{m;i,g}^{(2)}, E_{m;g,j}^{(r)}] & = [\dot{F}_{m;i,g}^{(2)}, \dot{E}_{m;g,j}^{(r)}]-\sum_{k=1}^{\beta}[\dot{F}_{m;i,g}^{(2)}, \dot{E}_{m;g,k}^{(r-1)}]\tilde{e}_{N-\beta+k,N-\beta+j}\\
& +\big[[F_{m;i,g}^{(2)},E_{m;g,h}^{(r-1)}],\tilde{e}_{N-2\beta+h,N-\beta+j}\big]\\
& = -\dot{D}_{m+1;i,j}^{(r+1)}-\sum_{t=1}^{r+1}\dot{D}_{m+1;i,j}^{(r+1-t)}\dot{D}_{m;g,g}^{\prime(t)}\\
& +\sum_{k=1}^{\beta} \dot{D}_{m+1;i,k}^{(r)}\tilde{e}_{N-\beta+k,N-\beta+j}
+\sum_{k=1}^\beta\sum_{t=1}^{r} \dot{D}_{m+1;i,k}^{(r-t)}\dot{D}_{m;g,g}^{\prime(t)}\tilde{e}_{N-\beta+k,N-\beta+j}\\
& -[\dot{D}_{m+1;i,h}^{(r)},\tilde{e}_{N-2\beta+h,N-\beta+j}]
-\sum_{t=1}^{r}[\dot{D}_{m+1;i,h}^{(r-t)}\dot{D}_{m;g,g}^{\prime(t)},\tilde{e}_{N-2\beta+h,N-\beta+j} ]\\
& =-\sum_{t=1}^{r}\dot{D}_{m+1;i,j}^{(r+1-t)}\dot{D}_{m;g,g}^{\prime(t)}
 +\sum_{t=1}^{r}\sum_{k=1}^{\beta}\dot{D}_{m+1;i,k}^{(r-t)}\dot{D}_{m;g,g}^{\prime(t)}\tilde{e}_{N-\beta+k,N-\beta+j}\\
& -\sum_{t=1}^{r}[\dot{D}_{m+1;i,h}^{(r-t)}\dot{D}_{m;g,g}^{\prime(t)},\tilde{e}_{N-2\beta+h,N-\beta+j} ]
-D_{m+1;i,j}^{(r+1)} - \dot{D}_{m+1;i,j}^{(0)}\dot{D}_{m;g,g}^{\prime(r+1)}\\
& = - D_{m+1;i,j}^{(r+1)} -\sum_{t=1}^{r+1} D_{m+1;i,j}^{(r+1-t)}\dot{D}_{m;g,g}^{\prime(t)}
\end{align*}
Thus $D_{m+1;i,j}^{(r+1)}=-[F_{m;i,g}^{(2)}, E_{m;g,j}^{(r)}]-\sum_{t=1}^{r+1} D_{m+1;i,j}^{(r+1-t)}\dot{D}_{m;g,g}^{\prime(t)}$ for $m>1$.\\

The case for $m=1$ is exactly the same, except that the $E_1$'s and $F_1$'s are odd elements.
\end{proof}

\begin{lemma}
Suppose $s^\mu_{m,m+1}=1$. Then 
\begin{enumerate}
  \item $D_{m+1;i,j}^{(r)}$ are $\mathfrak{m}$-invariant for all $r\geq 0$ and $1\leq i,j\leq \mu_{m+1}$. 
  \item $E_{m;i,j}^{(r)}$ are $\mathfrak{m}$-invariant for all $r>1$ and $1\leq i\leq \mu_{m}$, $1\leq j\leq \mu_{m+1}$.
  \end{enumerate}
\end{lemma}
\begin{proof}
These elements are known to be $\dot{\mathfrak{m}}$-invariant by Lemma \ref{demdot}. Hence it suffices to show that they are $\mathfrak{\dot{m}^c}$-invariant. By Lemma \ref{de12dotminv}, Lemma \ref{derdotminv} and induction on $r$, the statement follows.
\end{proof}

\begin{lemma}\label{des2minv}
Suppose that $s_{m,m+1}^\mu>1$. Then the following elements are invariant under the twisted action of $\tilde{e}_{N-\beta+f,N-2\beta+g}$ for all $1\leq f,g\leq\beta$.
\begin{enumerate}
\item $D_{m+1;i,j}^{(r)}$ for all $r\geq 2$ and $1\leq i,j\leq \mu_{m+1}$.
\item $E_{m;i,j}^{(r)}$ for all $r>s_{m,m+1}^\mu$ and $1\leq i\leq \mu_m$, $1\leq j\leq \mu_{m+1}$.
\end{enumerate}
\end{lemma}

\begin{proof}
(1) Let $\ddot{\pi}$ be the pyramid obtained by deleting the rightmost two columns of $\pi$. 
Define $\ddot{\mathfrak{p}}$, $\ddot{\mathfrak{m}}$ and $\ddot{e}\in\mathfrak{gl}_{M|N-2\beta}$ as before, and embed $U(\ddot{\mathfrak{g}})$ into $U(\dot{\mathfrak{g}})$ as we embed $U(\dot{\mathfrak{g}})$ into $U(\mathfrak{g})$. By the induction hypothesis, the elements $\ddot{D}_{m+1;i,j}^{(r)}$ in $\W_{\ddot\pi}$ are $\ddot{\mathfrak{m}}$-invariant.

Assuming $r\geq 2$ and applying Lemma \ref{sbr11} to $\pi$, we have
\begin{equation}\label{des2minv1}
D_{m+1;i,j}^{(r)} = \dot D_{m+1;i,j}^{(r)}
-\sum_{k=1}^\beta
\dot D_{m+1;i,k}^{(r-1)} \tilde e_{N-\beta+k,N-\beta+j}
+
[\dot D_{m+1;i,h}^{(r-1)}, \tilde e_{N-2\beta+h,N-\beta+j}]
\end{equation}
 
Applying Lemma \ref{sbr11} to $\dot{\pi}$, we obtain
\begin{equation}\label{des2minv2}
\dot{D}_{m+1;i,j}^{(r)} = \ddot{D}_{m+1;i,j}^{(r)}
-\sum_{k=1}^\beta
\ddot{D}_{m+1;i,k}^{(r-1)} \tilde e_{N-2\beta+k,N-2\beta+j}
+
[\ddot{D}_{m+1;i,h}^{(r-1)}, \tilde e_{N-3\beta+h,N-2\beta+j}]
\end{equation}

Substituting (\ref{des2minv2}) into (\ref{des2minv1}) and simplifying by (\ref{etilrel}), one deduces that for all $r\geq 2$,
$D_{m+1;i,j}^{(r)}=A-B+C-D+E-F-G+H$, where
\begin{align*}
A&=\ddot{D}_{m+1;i,j}^{(r)}, & 
B&=\sum_{k=1}^{\beta}\ddot{D}_{m+1;i,k}^{(r-1)}\tilde{e}_{N-2\beta+k,N-2\beta+j},\\
C&=[\ddot{D}_{m+1;i,h}^{(r-1)},\tilde{e}_{N-3\beta+h,N-2\beta+j}], & 
D&=\sum_{k=1}^\beta\ddot{D}_{m+1;i,k}^{(r-1)}\tilde{e}_{N-\beta+k,N-\beta+j}\\
E&=\sum_{k,s=1}^\beta \ddot{D}_{m+1;i,s}^{(r-2)}\tilde{e}_{N-2\beta+s,N-2\beta+k}\tilde{e}_{N-\beta+k,N-\beta+j}, &
F&=\sum_{k=1}^\beta \ddot{D}_{m+1;i,k}^{(r-2)}\tilde{e}_{N-2\beta+k,N-\beta+j},\\
G&=\sum_{k=1}^\beta[\ddot{D}_{m+1;i,h}^{(r-2)},\tilde{e}_{N-3\beta+h,N-2\beta+k}]\tilde{e}_{N-\beta+k,N-\beta+j},&
H&=[\ddot{D}_{m+1;i,g}^{(r-2)},\tilde{e}_{N-3\beta+g,N-\beta+j}].
\end{align*}

Let $x=\tilde{e}_{N-\beta+f,N-2\beta+g}$ for some $1\leq f,g\leq \beta$. Note that $x$ commutes with all elements in $U(\ddot{\mathfrak{p}})$. Applying ad $x$ to the above elements and using (\ref{rhodef}), (\ref{etilrel}) and (\ref{chidef}),
we obtain their images under $\text{pr}_{\chi}$ as follows:
\begin{align*}
\text{pr}_{\chi}([x,A])&=0,\qquad \text{pr}_{\chi}([x,B])=\del_{fj}\ddot{D}_{m+1;i,g}^{(r-1)},\\
\text{pr}_{\chi}([x,C])&=0,\qquad \text{pr}_{\chi}([x,D])=-\del_{fj}\ddot{D}_{m+1;i,g}^{(r-1)},\\
\text{pr}_{\chi}([x,E])&=-\beta\del_{fj}\ddot{D}_{m+1;i,g}^{(r-2)}+\ddot{D}_{m+1;i,g}^{(r-2)}\tilde{e}_{N-\beta+f,N-\beta+j}\\
&\qquad\qquad\qquad\qquad\qquad
-\del_{fj}\sum_{k=1}^{\beta}\ddot{D}_{m+1;i,k}^{(r-2)}\tilde{e}_{N-2\beta+k,N-2\beta+g},\\
\text{pr}_{\chi}([x,F])&=-\beta\del_{fj}\ddot{D}_{m+1;i,g}^{(r-2)}+\ddot{D}_{m+1;i,g}^{(r-2)}\tilde{e}_{N-\beta+f,N-\beta+j}\\
&\qquad\qquad\qquad\qquad\qquad
-\del_{fj}\sum_{k=1}^{\beta}\ddot{D}_{m+1;i,k}^{(r-2)}\tilde{e}_{N-2\beta+k,N-2\beta+g},\\
\text{pr}_{\chi}([x,G])&=-\del_{fj}[\ddot{D}_{m+1;i,h}^{(r-2)},\tilde{e_{N-3\beta+h,N-2\beta+g}}],\\
\text{pr}_{\chi}([x,H])&=-\del_{fj}[\ddot{D}_{m+1;i,h}^{(r-2)},\tilde{e}_{N-3\beta+h,N-2\beta+g}].
\end{align*}
As a result, $\text{pr}_{\chi}([x,D_{m+1;i,j}^{(r)}])=0$.\\
(2) can be proved by a similar method.
\end{proof}

\begin{proposition}\label{minvprop}
The elements
\begin{align*}
&\{D_{a;i,j}^{(r)}\}_{1 \leq a \leq m+1,1 \leq i,j \leq \mu_a, r > 0},\\
&\{E_{a;i,j}^{(r)}\}_{1 \leq a < m+1, 1 \leq i \leq \mu_a, 1 \leq j \leq \mu_{a+1}, r > s_{a,b}^\mu },\\
&\{F_{a;i,j}^{(r)}\}_{1 \leq a < m+1, 1 \leq i \leq \mu_{a+1}, 1 \leq j \leq \mu_{a}, r > s_{b,a}^\mu}
\end{align*}
of $U({\mathfrak{p}})$ are $\mathfrak{m}$-invariant.
\end{proposition}

\begin{proof}
By the induction hypothesis and Lemma \ref{demint1}--Lemma \ref{des2minv}.
\end{proof}

Proposition~\ref{minvprop} implies that the elements in the description of Theorem~\ref{main} are indeed elements of $\W_\pi$. By the induction hypothesis, we may identify $Y_{1|n}^{\ell-1}(\dot\sigma)$ with $\W_{\dot\pi}\subseteq U(\dot{\mathfrak{p}})$ and hence the generators $\dot D_{a:i,j}^{(r)}$, $\dot E_{a;i,j}^{(r)}$ and $\dot F_{a;i,j}^{(r)}$ in $Y_{1|n}^{\ell-1}(\dot\sigma)$ coincide with the elements of $\W_{\dot \pi}$ with the same name.  Recall the injective superalgebra homomorphism $\Delta_R:Y_{1|n}^\ell(\sigma)\rightarrow U(\dot{\mathfrak{p}})\otimes U(\gl_\beta)$ in Theorem \ref{PBWSYLpara}.

By Corollary \ref{dimcoro}, for each $d\geq 0$, we have
\begin{equation}\label{dim1}
\dim \Delta_R (F_dY_{1|n}^\ell(\sigma))=\dim F_d Y_{1|n}^\ell(\sigma)=\dim F_d S(\g^e),
\end{equation}
where $F_dS(\g^e)$ is the sum of all graded elements in $S(\g^e)$ of degree $\leq d$ in the Kazhdan grading.

Define the general parabolic generators $E_{a,b;i,j}^{(r)}$ and $F_{ab,;i,j}^{(r)}$ in $F_r U(\mathfrak{p})$ by formulae (\ref{eparag}) and (\ref{fparag}) recursively, choosing an arbitrary integer $k$ there.
Let $X_d$ denote the subspace of $U(\mathfrak{p})$ spanned by all supermonomials in the elements
\begin{align*}
&\{D_{a;i,j}^{(r)}\}_{1\leq a\leq m+1, 1\leq i,j\leq \mu_a, 0\leq r\leq s_{a,a}^{\mu}},\\
&\{E_{a,b;i,j}^{(r)}\}_{1\leq a<b\leq m+1, 1\leq i\leq \mu_a, 1\leq j\leq \mu_b, s_{a,b}^\mu<r\leq s_{a,b}^{\mu}+p_a^\mu},\\
&\{F_{a,b;i,j}^{(r)}\}_{1\leq b<a\leq m+1, 1\leq i\leq \mu_a, 1\leq j\leq \mu_b, s_{a,b}^\mu<r\leq s_{a,b}^\mu+p_a^\mu} .
\end{align*}
taken in some fixed order and of total degree $\leq d$. By Proposition \ref{minvprop}, $X_d$ is a subspace of $F_d \W_\pi$.

Define a superalgebra homomorphism $\psi_R:U(\mathfrak{p})\rightarrow U(\dot{\mathfrak{p}})\otimes U(\gl_\beta)$ by
\begin{equation}\notag
\psi_R(\tilde e_{i,j}):= \left\{
\begin{array}{ll}
\tilde e_{i,j}\otimes 1 &\hbox{if $\col(i)\leq \col (j)\leq \ell-1$,}\\
0 &\hbox{if $\col(i)\leq \ell-1, \col(j)=\ell$,}\\
1\otimes \tilde e_{i-N+\beta,j-N+\beta} &\hbox{if $\col(i)=\col(j)=\ell$.}
\end{array}
\right.
\end{equation}
By Lemma \ref{sbabyr1}, we have 
\begin{align*}
&\psi_R(D_{a;i,j}^{(r)})=\dot D_{a;i,j}^{(r)}\otimes 1-\delta_{a,m+1}\sum_{k=1}^\beta \dot D_{a;i,k}^{(r-1)}\otimes \tilde e_{k,j},\\
&\psi_R(E_{a;i,j}^{(r)})=\dot E_{a;i,j}^{(r)}\otimes 1-\delta_{a,m}\sum_{k=1}^\beta \dot E_{a;i,k}^{(r-1)}\otimes \tilde e_{k,j},\\
&\psi_R(F_{a;i,j}^{(r)})=\dot F_{a;i,j}^{(r)}\otimes 1.
\end{align*}

Comparing this with Theorem \ref{baby1}$({\it 1})$ and recalling the PBW basis for $Y_{1|n}^\ell(\sigma)$ obtained from Corollary \ref{pbwbasis}, we deduce that $\psi_R(X_d)=\Delta_R(F_dY_{1|n}^\ell(\sigma))$. Combining this with (\ref{dim1}) and Corollary \ref{dimcoro}, we obtain 
\begin{equation*}
\dim F_dS(\g^e)=\dim \psi_R(X_d)\leq \dim X_d \leq \dim F_d\W_{\pi}\leq \dim F_dS(\g^e). 
\end{equation*}

Thus equality holds everywhere so we have $X_d=F_d\W_\pi$ for each $d\geq 0$, and in particular, the map $\psi_R:\W_\pi\rightarrow U(\dot{\mathfrak{p}})\otimes \gl_\beta$ is an injective homomorphism. Moreover, recall the map $\Delta_R:Y_{1|n}^\ell(\sigma)\rightarrow U(\dot{\mathfrak{p}})\otimes \gl_\beta$ defined in Theorem \ref{baby1}$({\it 1})$.
Comparing the formulae, we have that $\psi_R(D_{a;i,j}^{(r)})=\Delta_R(D_{a;i,j}^{(r)})$, where the elements $D_{a;i,j}^{(r)}$ on the left-hand side are the elements of $\W_\pi$, and the elements $D_{a;i,j}^{(r)}$ on the right-hand side are the generators of $Y_{1|n}^\ell(\sigma)$. Similarly, $\psi_R(E_{a;i,j}^{(r)})=\Delta_R(E_{a;i,j}^{(r)})$ and $\psi_R(F_{a;i,j}^{(r)})=\Delta_R(F_{a;i,j}^{(r)})$ for all admissible $a,i,j,r$.

Therefore, the composition map $\psi_R^{-1}\circ \Delta_R:Y_{1|n}^\ell(\sigma)\rightarrow \W_\pi$ is exactly the filtered superalgebra isomorphism described in Theorem \ref{main} and the elements listed in Theorem \ref{main} indeed generate $\W_\pi$. This proves Theorem \ref{main} in the case R.

Next we sketch how to complete the induction step in the case L. In this case, we enumerate the bricks of $\pi$ down columns {\em from right to left}. Note that different ways of enumerating are just choosing different bases to describe $\gl_{M|N}\cong \End (\mathbb{C}^{M|N})$ so we may choose the way most suitable for our current purpose.

Let $\dot \pi$ denote the pyramid obtained from $\pi$ by deleting the {\em left-most} column of $\pi$; that is, deleting the bricks numbered with $N,N-1,\ldots,N-\beta+1$. Let $\dot\sigma$ be the shift matrix obtained from (\ref{babyl1}), where the corresponding pyramid is exactly $\dot\pi$, and define $\dot{\mathfrak{p}}, \dot{\mathfrak{m}}, \dot e \in \dot \g :=\mathfrak{gl}_{M|N-\beta}$ via (\ref{mpdef}) and (\ref{edef}) with respect to $\dot \pi$. Note that in this case we embed $U(\dot{\g})$ into $U(\g)$ by the natural embedding, since it already sends the elements $\tilde{e}_{ij}$ of $U(\dot{\g})$ to the elements $\tilde{e}_{ij}$ of $U(\g)$ for all $1\leq i,j\leq N-\beta$.

Under this embedding, the superalgebra $\W_{\dot{\pi}}=U(\dot{\mathfrak{p}})^{\dot{\mathfrak{m}}}$ is a subalgebra of $U(\dot{\mathfrak{p}})\subset U(\mathfrak{p})$ and the twisted action of $\dot{\mathfrak{m}}$ on $U(\dot{\mathfrak{p}})$ is exactly the restriction of the twisted action of $\mathfrak{m}$ on $U(\mathfrak{p})$. Let $\dot D_{a;i,j}^{(r)}, \dot{D}_{a;i,j}^{\prime(r)}$, $\dot E_{a;i,j}^{(r)}$ and $\dot F_{a;i,j}^{(r)}$ denote the elements of $U(\dot{\mathfrak{p}})$ as defined in \textsection \ref{Invariants} associated to the shape $\mu$ which is also admissible for $\dot\sigma$. By the induction hypothesis, these elements are $\dot{\mathfrak{m}}$-invariant.

The idea is exactly the same. By the following crucial lemma, which is the analogue of Lemma \ref{sbabyr1}, we may express the elements $D_{a;i,j}^{(r)}, {D}_{a;i,j}^{\prime(r)}$, $E_{a;i,j}^{(r)}$ and $F_{a;i,j}^{(r)}$ in $U(\mathfrak{p})$ in terms of $\dot D_{a;i,j}^{(r)}, \dot{D}_{a;i,j}^{\prime(r)}$, $\dot E_{a;i,j}^{(r)}$ and $\dot F_{a;i,j}^{(r)}$. Then by case-by-case discussions and computations, we can prove that all of the $D_{a;i,j}^{(r)}, {D}_{a;i,j}^{\prime(r)}$, $E_{a;i,j}^{(r)}$ and $F_{a;i,j}^{(r)}$ are indeed $\mathfrak{m}$-invariant. Since the arguments are almost identical, we will only provide the most crucial lemma below, where the proof is exactly the same.
\begin{lemma}\label{sbabyl1}
The following equations hold for $r > 0$, all admissible $a,i,j$ and any fixed $1 \leq h \leq \beta$:
\begin{align}\label{sbl1}\notag
D_{a;i,j}^{(r)} & = \dot D_{a;i,j}^{(r)}\\
 &+\delta_{a,m+1}
\left(-\sum_{k=1}^\beta \tilde e_{N-\beta+i,N-\beta+k}
\dot D_{a;k,j}^{(r-1)} 
+
[ \tilde e_{N-\beta+i,N-2\beta+h}, \dot D_{a;h,j}^{(r-1)}]\right),\\
E_{a;i,j}^{(r)} & = \dot E_{a;i,j}^{(r)},\\
F_{a;i,j}^{(r)} & = \dot F_{a;i,j}^{(r)} + \delta_{a,m}\left(
-\sum_{k=1}^\beta \tilde e_{N-\beta+i,N-\beta+k}\dot F_{a;k,j}^{(r-1)} 
+ [\tilde e_{N-\beta+i,N-2\beta+h}, \dot F_{a;h,j}^{(r-1)} ]\right),\label{sbl2}
\end{align}
where for (\ref{sbl2}) we are assuming that $r > s_{m+1,m}^\mu$ if $a=m$.
\end{lemma}

With the help of Lemma \ref{sbabyl1}, one can deduce that the statement of Proposition~\ref{minvprop} still holds in the case L. Finally, define a superalgebra homomorphism $\psi_L:U(\mathfrak{p})\rightarrow  U(\gl_\beta)\otimes U(\dot{\mathfrak{p}})$ by
\begin{equation}\notag
\psi_L(\tilde e_{i,j}):= \left\{
\begin{array}{ll}
\tilde e_{i-N+\beta,j-N+\beta}\otimes 1 &\hbox{if $\col(i)=\col (j)=1$,}\\
0 &\hbox{if $\col(i)=1, \col(j)\geq 2$,}\\
1\otimes \tilde e_{i,j} &\hbox{if $2\leq \col(i)\leq \col(j)$.}
\end{array}
\right.
\end{equation}
By Lemma \ref{sbabyl1}, we have that
\begin{align*}
&\psi_L(D_{a;i,j}^{(r)})=1\otimes \dot D_{a;i,j}^{(r)}-\delta_{a,m+1}\sum_{k=1}^\beta \tilde e_{i,k} \otimes \dot D_{a;k,j}^{(r-1)}\\
&\psi_L(E_{a;i,j}^{(r)})=1\otimes \dot E_{a;i,j}^{(r)},\\
&\psi_L(F_{a,ji,j}^{(r)})=1\otimes \dot F_{a;i,j}^{(r)}-\delta_{a,m}\sum_{k=1}^\beta \tilde e_{i,k} \otimes\dot F_{a;k,j}^{(r-1)}.
\end{align*}
Exactly the same argument as in the case R shows that the map $\psi_L$ is injective and the composition map $\psi_L^{-1}\circ \Delta_L:Y_{1|n}^\ell(\sigma)\rightarrow \W_\pi$ gives the required isomorphism of filtered superalgebras. This completes the proof of Theorem \ref{main}.

\begin{corollary}\label{ezg}
Let $\pi$ be a signed pyramid satisfying (\ref{crucialhypo}) and $\vec\pi$ be another signed pyramid obtained by horizontally shifting rows of $\pi$. Let $\W_\pi$ and $\W_{\vec\pi}$ denote the associated finite $W$-superalgebras, respectively. Then there exists a superalgebra isomorphism $\iota:\W_\pi\rightarrow \W_{\vec\pi}$ defined on parabolic generators with respect to an admissible shape $\mu$ by (\ref{iotadef}).
\end{corollary}
\begin{proof}
Follows from Theorem \ref{main} and (\ref{iotaiso}).
\end{proof}

\begin{remark}
In the classical (non-super) case, the definition of a finite $W$-algebra is independent of the choices of the good $\mathbb{Z}$-gradings \cite{BG}. Under certain mild assumption, which is satisfied in our current case, such a phenomenon is generalized to the Lie superalgebra case \cite[Theorem~3.7, Remark~3.11]{Zh}, and our Corollary \ref{ezg} is included as a special case.
\end{remark}

\subsection*{Acknowledgements}
Parts of the results in this article were established when the author was a graduate student in University of Virginia. The author is grateful to his adviser Weiqiang Wang for his patient guidance. The author would also like to thank the anonymous reviewer for his/her valuable comments and suggestions to improve the quality of this article. The author is supported by post-doctorial fellowship of Institution of Mathematics, Academia Sinica, Taipei, Taiwan.

\end{document}